\documentclass[10pt,twoside]{amsart}
\usepackage[margin=1.1in]{geometry}

\usepackage{amssymb,amsthm,amsmath,latexsym,stmaryrd}
\usepackage[all]{xy}

\theoremstyle{plain}
\newtheorem{definition}{Definition}[section]

\newtheorem{lemma}[definition]{Lemma}
\newtheorem{df}{Definition}[subsection]
\newtheorem{ex}[df]{Example}
\newtheorem{rk}[df]{Remark}

\theoremstyle{plain}
\newtheorem{proposition}[definition]{Proposition}
\newtheorem{remark}[definition]{Remark}
\newtheorem{theorem}[definition]{Theorem}
\newtheorem{lemme}[definition]{Lemme}

\newcommand{\Hom}{{\rm Hom}}

\title{Deformations of linear Poisson orbifolds}
\author{Gilles Halbout}
\address{Institut de Math\'ematiques et de Mod\'elisation de Montpellier,
Universit\'e de Montpellier 2, CC5149, Place Eug\`ene Bataillon,
F-34095 Montpellier CEDEX 5, France}
\email{halbout@@math.univ-montp2.fr}

\author{Jean-Michel Oudom}
\address{Institut de Math\'ematiques et de Mod\'elisation de Montpellier,
Universit\'e de Montpellier 2, CC5149, Place Eug\`ene Bataillon,
F-34095 Montpellier CEDEX 5, France}
\email{oudom@math.univ-montp2.fr}

\author{Xiang Tang}
\address{Department of Mathematics, Washington University, St. Louis, Missouri, USA, 63130}
\email{xtang@math.wustl.edu}

\begin{document}

\begin{abstract}
Let $\Gamma$ be a finite group acting faithfully and linearly on a
vector space $V$. Let $T(V)$ ($S(V)$) be the tensor (symmetric)
algebra associated to $V$ which has a natural $\Gamma$ action. We
study generalized quadratic relations on the tensor algebra
$T(V)\rtimes \Gamma$. We prove that the quotient algebras of
$T(V)\rtimes \Gamma$ by such relations satisfy PBW property. Such
quotient algebras can be viewed as quantizations of linear or
constant Poisson structures on $S(V)\rtimes \Gamma$, and are natural
generalizations of symplectic reflection algebras.
\end{abstract}
\maketitle

\section{Introduction}

Poisson structure on a manifold is a bivector field $\pi$ whose
Schouten-Nijenhuis bracket with itself vanishes, i.e. $\pi\in
\Gamma(\wedge^2 TM)$, and $[\pi, \pi]=0$. The problem of deformation
quantization of a Poisson manifold was solved by Kontsevich in his
semina paper \cite{K}. In this paper, we study quantization problem
of Poisson structures on an orbifold following \cite{H-T}.

The study of the first and third author in \cite{H-T} starts with
the idea that Poisson structures on an algebra $A$ should correspond
to infinitesimal deformations of $A$. According to Gerstenhaber's
theory, an infinitesimal deformation of an algebra is classified by
a second Hochschild cohomology class in $H^2(A,A)$ whose
Gerstenhaber bracket with itself is zero. This type of cohomology
class is called a Poisson structure (\cite{B-G-p}, \cite{xu}) on
$A$. Applying this idea to orbifold, we can represent an orbifold
$X$ by a proper \'etale groupoid $\mathcal {G}$ \cite{M} (different
representations of same orbifold are Morita equivalent as Lie
groupoids). We consider the smooth groupoid algebra
$C^\infty_c(\mathcal {G})$ associated to $\mathcal {G}$. We studied
in \cite{H-T} Poisson structures on $C^\infty_c(\mathcal {G})$. We
find that Poisson structures on $C^\infty_c(\mathcal {G})$ are
richer than we naturally expect from geometry. On an orbifold,
multivector fields and Schouten-Nijenhuis bracket are well defined.
Accordingly, we can consider bivector fields on $X$ whose
Schouten-Nijenhuis bracket with themselves vanish. We found that
\cite{H-T}[Theorem 4.1] there are many more Poisson structures on
$C^\infty_c(\mathcal {G})$ than the above type of bivector fields on
$X$. For example, in the case of a finite group $\Gamma$ acting on a
symplectic vector space $V$, we \cite{H-T}[Corollary 4.2] find
Poisson structures on $S(V^*)\rtimes \Gamma$ which have supports on
codimension 2 fixed point subspaces, where $S(V^*)$ is the algebra
of real coefficients symmetric
polynomials on the dual vector space $V^*$.\\

In this paper, we continue our study of Poisson structures in the
above framework. We will study Poisson structures in a neighborhood
of a point in a reduced orbifold.  Locally, a reduced orbifold can
always be viewed as a quotient of a finite group acting faithfully
and linearly on an open set of $\mathbb {R}^n$. This leads us to
study the following data. Let $\Gamma$ be a finite group acting on a
vector space $V$ faithfully, and $S(V^*)$ be the algebra of
polynomials on $V^*$. The $\Gamma$ action on $V^*$ defines the
crossed product algebra $S(V^*)\rtimes \Gamma$. According to
\cite{N-P-P-T}, the second Hochschild cohomology $H^2(S(V^*)\rtimes
\Gamma, S(V^*)\rtimes \Gamma)$ is equal to
\[
(S(V^*)\otimes \wedge^2V)^\Gamma\oplus( \sum_{\gamma\in \Gamma,
l(\gamma)=2} S({V^\gamma}^*)\otimes \wedge^2 N^\gamma)^\Gamma.
\]
In the above equation,  $l(\gamma)$ is the codimension of the
$\gamma$ fixed point subspace $V^\gamma$ and $N^\gamma$ consists of
vectors in $V$ vanishing on ${V^\gamma}^*$ (the fixed points
subspace of $\gamma$ action on $V^*$), and $\Gamma$ acts on the set
$S:=\{\gamma\in \Gamma, l(\gamma)=2\}$ by conjugation.

We study two types of Poisson structures on $S(V^*)\rtimes \Gamma$
which are of the forms
\[
i)\ \Hom(\wedge^2V^*, \mathbb{R}\Gamma),\qquad\ ii)\ \Hom(\wedge
^2V^*, V^*\otimes_{\mathbb{R}} \mathbb{R}\Gamma).
\]
The first type of Poisson structure can be viewed as constant value
Poisson structures, and the second type can be viewed as linear
Poisson structures, which define generalized ``Lie" algebra
structures on $V^*\otimes_{\mathbb R} \mathbb R \Gamma$. Our main
theorems for these Poisson structures are that the quotient algebras
of $T(V^*)\rtimes \Gamma$ by the relations defined by the above two
types of Poisson structures satisfy PBW property. (For general
linear Poisson structures, we need to assume that $V$ is equipped
with a $\Gamma$-invariant complex structure.) The way we prove such
theorems is using the Braverman-Gaitsory conditions \cite{B-G} for
PBW property. However, the proof for the second type is quite
involved. We need to use properties of finite subgroups of $GL(2,
\mathbb{C})$, which is closely related to McKay correspondence. A
new and interesting phenomena we found in the proof of PBW property
is that we have to have a nontrivial coboundary term for the bracket
$[\pi, \pi]$ of the linear Poisson structure $\pi$. This kind of
term never shows up in the study of PBW property for Lie algebras
and symplectic reflection algebras. The PBW property of the quotient
algebras shows that they define quantizations of the Poisson
structures on $S(V^*)\rtimes \Gamma$. This confirms that any
constant or linear Poisson structures on $S(V^*)\rtimes \Gamma$ can
be quantized, and gives a strong evidence that the deformation
theory of the algebra $S(V^*)\rtimes \Gamma$ is formal.

The second part of this paper is dedicated to studying various
properties and examples of the above two types of Poisson structures
and their quantizations. We mention a few of them here. Firstly,
using Poisson cohomology computation, we are able to give a new
computation of Hochschild cohomology of a symplectic reflection
algebra \cite{E-G}[Theorem 1.8]. The advantage of our work is that
our result works for Laurent series of $\hbar$ so that we can drop
the assumption ``except possibly a countable set" in
\cite{E-G}[Theorem 1.8]. Secondly, assuming a finite group $\Gamma$
acts faithfully and linearly on a Lie algebra $\mathfrak {g}$, we
compute the Hochschild cohomology of $H^\bullet(\mathcal U(\mathfrak
g)\rtimes \Gamma, \mathcal U(\mathfrak g)\rtimes \Gamma)$ with
$\mathcal U(\mathfrak g)$ the universal enveloping algebra of
$\mathfrak g$. This is a natural generalization of results
\cite{A-F-L-S} for a finite group action on a symplectic vector
space. Our result shows that the Hochschild cohomology
$H^\bullet(\mathcal U(\mathfrak g)\rtimes \Gamma, \mathcal
U(\mathfrak g)\rtimes \Gamma)$ is computed by the noncommutative
Poisson cohomology associated to the $\Gamma$ action and Lie-Poisson
structure $\pi$ on $\mathfrak g$. Thirdly, if $V$ is equipped with a
$\Gamma$ invariant Lie Poisson structure, we introduce a large class
of examples of linear Poisson structures on $S(V^*)\rtimes \Gamma$.
Quantizations of these linear Poisson structures should be viewed as
natural generalizations of symplectic reflection algebras. Analogous
to the symplectic case, we are able to prove that the Poisson
cohomology of such a general linear Poisson structure is determined
by the data supported at the identity of $\Gamma$ in the case that
$\Gamma$ is abelian.

In the third part of this paper, we restrict ourselves to
$\mathbb{R}^2$ with a cyclic group $\mathbb Z_n$ action. Let
$\omega$ be the standard symplectic form on $\mathbb{R}^2$ and $\pi$
be the corresponding Poisson structure. We study the quantization of
the Poisson structure $\pi_\gamma: x\wedge y\to \pi(x,y)\gamma$ on
$S({\mathbb {R}^2})\rtimes \mathbb Z_n$ for $\gamma\in \mathbb Z_n$
(We identify ${\mathbb R^2}^*$ with $\mathbb R^2$). Nadaud \cite{N}
gave a Moyal type formula for such a deformation quantization. Here,
we use this formula to study the center of the quantization. Our
computation shows that the center of the quantization of
$\pi_\gamma$ is not isomorphic to the center of the algebra
$S(\mathbb{R}^2)\rtimes \mathbb {Z}_n$. Instead, the center of the
quantization is a nontrivial deformation of the center of
$S(\mathbb{R}^2)\rtimes \mathbb {Z}_n$. This suggests that there is
no analog of Duflo's isomorphism for the quantization of the Poisson
structure $\pi_\gamma$, and also that deformation quantization of
$\pi_\gamma$ on $S(\mathbb{R}^2)\rtimes \mathbb {Z}_n$ is closely
connected to the deformation of the underlying orbifold singularity.
We plan to study the relation between deformations of $S(V^*)\rtimes
\Gamma$ and deformations of
the underlying orbifold $V/\Gamma$ in the near future.\\

This paper is organized as follows. In the second section, we review
some results about Hochschild cohomology $H^\bullet(S(V^*)\rtimes
\Gamma, S(V^*)\rtimes \Gamma)$ in \cite{N-P-P-T} and \cite{H-T}, and
also the Braverman-Gaitsgory conditions for PBW property \cite{B-G};
in the third section, we prove that constant and linear Poisson
structures on $S(V^*)\rtimes \Gamma$ can be quantized; in the forth
section, we study various properties and examples of constant and
linear Poisson structures on $S(V^*)\rtimes \Gamma$; in the fifth
section, using Nadaud's formula, we study the centers of
quantizations of some Poisson structures on $S({\mathbb
{R}^2})\rtimes
\mathbb{Z}_n$.\\

\noindent{\bf Acknowledgements:} We would like to thank C\'edric
Bonnaf\'e for discussion about finite subgroups of $GL(2,
\mathbb{C})$,  Georges Pinczon for showing the results of Nadaud
\cite{N}, and Victor Ginzburg for general discussion of symplectic
reflection algebras and Duflo's isomorphism. The research of the
third author is partially supported by NSF Grant 0703775.

\section{Preliminaries and notations}

In this whole paper, $\Gamma$ is a finite group, acting faithfully
and linearly on a finite dimensional real vector space $V$. We fix
on $V$ a $\Gamma$-invariant metric. We denote $C(\Gamma)$ the set of
conjugacy classes of $\Gamma$. For any element $\gamma$ in $\Gamma$,
let $V^\gamma$ be the $\gamma$-invariant subspace of $V$, $N^\gamma$
be the subspace of $V$ orthogonal to $V^\gamma$ which is the direct
sum of all nontrivial representations of $G(\gamma)$ (the subgroup
of $\Gamma$ generated by $\gamma$), $l(\gamma)$ be the real
codimension of $V^\gamma$, and $Z(\gamma)$ the centralizer of
$\gamma$ in $\Gamma$. In this paper, we always work with the field
$\mathbb {R}$. All dimensions, algebras, and tensor products if not
specified are over the field $\mathbb R$. For the convenience of
proofs, we are many times using the following complexification
trick,
\begin{equation}\label{eq:complex}
S_{\mathbb R}(V^*)\rtimes _{\mathbb R} \mathbb C \Gamma\cong
S_{\mathbb C}(V^*\otimes \mathbb C)\rtimes_{\mathbb C}\mathbb
C\Gamma \cong \Big(S_{\mathbb R}(V^*)\rtimes _{\mathbb R} \mathbb
R\Gamma \Big)\otimes_{\mathbb R} \mathbb C.
\end{equation}
This helps us to deduce the results in $\mathbb R$ from their
complex versions where $\gamma$ action is diagonalizable for any
$\gamma\in \Gamma$. Many results in this paper hold true for field
$\mathbb {C}$ and even more general field with characteristic 0 in
which the order of $\Gamma$ is invertible.

\subsection{The Koszul complex and the Hochschild cohomology of $S (V^*)\rtimes\Gamma$}
\label{Koszul}

The algebra $S(V^*)\rtimes \Gamma$ is generated by $V^*$ and
$\Gamma$ with the quadratic relations :
$$
x\otimes y \otimes \gamma - y\otimes x\otimes \gamma,\ \gamma
\otimes x -^\gamma x\otimes \gamma,
$$
for all $x$ and $y$ in $V^*$ and $\gamma$ in $\Gamma$, and $^\gamma
x$ is the image of $x$ under the $\gamma$ action. Moreover,
$S(V^*)\rtimes \Gamma$ is a Koszul algebra over the semi-simple
algebra $\mathbb{R}\Gamma$. The general theory of Koszul algebras
over a semi-simple algebra gives therefore a small complex which
calculates the Hochschild cohomology of $S(V^*)\rtimes \Gamma$ :
$$
CK^\bullet(S(V^*)\rtimes \Gamma)= \Big(\bigoplus_{\gamma\in \Gamma}
S(V^*)\otimes \wedge ^\bullet V\Big)^\Gamma.
$$

A $n$-cochain $f$ of this complex splits in a sum of maps $f_\gamma$
in $S(V^*)\otimes \Lambda^n V$. The $\Gamma$-invariance can be
written :
\begin{equation}\label{invariance}
^gf_\gamma(^{g^{-1}}\!\!x_1,\cdots , ^{g^{-1}}\!\!x_n)=f_{g\gamma
g^{-1}}(x_1,\cdots,x_n),
\end{equation}
which explains that $CK^\bullet(S(V^*)\rtimes \Gamma)$ splits in a
sum of sub-complexes :
$$
CK^\bullet(S(V^*)\rtimes \Gamma)=\bigoplus_{\gamma\in
C(\Gamma)}(S(V^*)\otimes \wedge ^\bullet V)^{Z(\gamma)}
$$
with the boundary
$$\partial_\gamma(f)(x_0,\cdots,x_n) =
\sum_{i=0}^n (-1)^{i} f (x_0,\cdots,\widehat{x}_i,\cdots,x_n)
(x_i-^\gamma\!x_i),
$$
for $x_0, \cdots, x_n\in V^*$.

Using this small complex,  Neumaier, Pflaum, Posthuma and the third
author calculated in \cite{N-P-P-T}  the Hochschild cohomology of
$S(V^*)\rtimes \Gamma$ :
\begin{equation}\label{N-P-P-T}
H^\bullet(S(V^*)\rtimes \Gamma,S(V^*)\rtimes \Gamma
)=\bigoplus_{\gamma\in
C(\Gamma)}\Big(S({V^\gamma}^*)\otimes\Lambda^{\bullet-l(\gamma)}V^\gamma
\otimes \Lambda^{l(\gamma)}N^\gamma \Big)^{Z(\gamma)}.
\end{equation}
This statement in $\mathbb R$ is easily deduced from its complex
version \cite{N-P-P-T} using the trick (\ref{eq:complex}).

We point out that the projection $\textsf{pr}_\gamma: (S(V^*)\otimes
\wedge^\bullet V)^{Z(\gamma)}\rightarrow (S({V^\gamma}^*)\otimes
\Lambda^{\bullet-l(\gamma)}V^\gamma \otimes
\Lambda^{l(\gamma)}N^\gamma)^{Z(\gamma)}$ and the embedding
$\iota:(S({V^\gamma}^*)\otimes \Lambda^{\bullet-l(\gamma)}V^\gamma
\otimes \Lambda^{l(\gamma)}N^\gamma)^{Z(\gamma)}\rightarrow
(S(V^*)\otimes \wedge^\bullet V)^{Z(\gamma)}$ are inverse
quasi-isomorphisms of complexes. Another useful remark is that if
$l(\gamma)$=dimension of $N^\gamma$ is odd, the determinant of
$\gamma$ action on $N^\gamma$ is -1 (otherwise $\gamma$ has an
eigenvalue 1 as $\gamma$ is an isometry). Therefore
$S({V^\gamma}^*)\otimes \Lambda^{\bullet-l(\gamma)}V^\gamma \otimes
\Lambda^{l(\gamma)}N^\gamma$ has no $\gamma$ invariant element if
$l(\gamma)$ is odd. Therefore, Poisson brackets on $S(V^*)\rtimes
\Gamma$ do not contain $\gamma$-component for $l(\gamma)=1$.
Furthermore, when $\Gamma$ acts faithfully, the identity of $\Gamma$
is the only group element with $l(\gamma)=0$.\\

We will say that a cocycle is constant if it is in $(\wedge
^2V)^\Gamma\oplus(\bigoplus_{\gamma\in \Gamma,
l(\gamma)=2}\wedge^2N^\gamma)^\Gamma$. Similarly, we will say that a
cocycle is linear if it is in $(V^*\otimes \wedge ^2V)^\Gamma
\oplus(\bigoplus_{\gamma\in \Gamma, l(\gamma)=2}{V^\gamma}^*\otimes
\wedge^2N^\gamma)^\Gamma$.

\subsection{The Braverman-Gaitsgory conditions for PBW}\label{BG}

Let $T(V^*)\rtimes\Gamma$ be the free $\mathbb R\Gamma$-algebra
generated by the bimodule $V^*\rtimes \Gamma$, and $A$ be its
quotient by the relations :
$$
x\otimes y - y \otimes x - \sum_\gamma \pi_\gamma(x,y)\gamma
-\sum_\gamma b_\gamma(x,y)\gamma
$$
where $\pi$ and $b$ are $\Gamma$-invariant elements in
$\oplus_{\gamma\in \Gamma}V^*\otimes \wedge^2(V)$  and
$\oplus_{\gamma\in \Gamma}\wedge^2 V$. As before, $\pi$ and $b$
split into sums of $\gamma$-components, and the $\Gamma$-invariance
is expressed in the same way as in (\ref{invariance}). The algebra
$A$ is clearly filtered by the length of words. Following Braverman
and Gaitsgory \cite{B-G}, the associated graded algebra $Gr(A)$ is
isomorphic to $S(V^*)\rtimes \Gamma$ if and only if :
\begin{equation}\label{BG1}
\partial_\gamma(\pi_\gamma)=0,
\end{equation}
\begin{equation}\label{BG2}
\sum_{\alpha\beta=\gamma}\pi_\alpha(\pi_\beta(x,y),z+^\beta\!z)+\pi_\alpha(\pi_\beta(y,z),x+^\beta\!x)+\pi_\alpha(\pi_\beta(z,x),y+^\beta\!y)=\partial_\gamma(b_\gamma),
\end{equation}
\begin{equation}\label{BG3}
\sum_{\alpha\beta=\gamma}b_\alpha(\pi_\beta(x,y),z+^\beta\!z)+b_\alpha(\pi_\beta(y,z),x+^\beta\!x)+b_\alpha(\pi_\beta(z,x),y+^\beta\!y)=0
\end{equation}
for all $\gamma$ in $\Gamma$ and $x$, $y$, $z$ in $V^*$.

When the three conditions above are satisfied, the algebra $A$ gives
a quantization of the algebra $S(V^*)\rtimes \Gamma$ for the same
reason as is explained in \cite{H-T}(Proposition 4.5). This will be
our method to obtain the quantization results of the next section.

For our purpose, let us denote $\llbracket
\pi,\pi\rrbracket_\gamma\in V^*\otimes \wedge^3V$ and  $\llbracket
b,\pi\rrbracket_\gamma\in \wedge^3V$ defined by :
$$
\llbracket\pi,\pi\rrbracket_\gamma(x,y,z):=\,\sum_{\alpha\beta=\gamma}\pi_\alpha(\pi_\beta(x,y),z+^\beta\!z)+\pi_\alpha(\pi_\beta(y,z),x+^\beta\!x)+\pi_\alpha(\pi_\beta(z,x),y+^\beta\!y),
$$
$$
\llbracket
b,\pi\rrbracket_\gamma(x,y,z):=\sum_{\alpha\beta=\gamma}b_\alpha(\pi_\beta(x,y),z+^\beta\!z)+b_\alpha(\pi_\beta(y,z),x+^\beta\!x)+b_\alpha(\pi_\beta(z,x),y+^\beta\!y)
$$
for all $x$, $y$, $z$ in $V^*$.

\subsection{The Gerstenhaber bracket on $H^
\bullet(S(V^*)\rtimes \Gamma, S(V^*)\rtimes \Gamma)$ and Poisson
structures}\label{Gerstenhaber}

The Gerstenhaber bracket on $H^\bullet(S(V^*)\rtimes \Gamma,
S(V^*)\rtimes \Gamma)$ was explicitly calculated by the first and
last authors in \cite{H-T}. We only recall here the results in the
cases we will need and refer to \cite{H-T} for a complete
description.

Firstly, the Gerstenhaber bracket of two constant cocycles is zero.

Secondly, let $b$ be a constant cocycle and $\pi$ be a linear
cocycle of $H^2(S(V^*)\rtimes \Gamma, S(V^*)\rtimes \Gamma)$. Let
$\textsf{pr}_\gamma$ be the projection from $S(V^*)\otimes \wedge
^\bullet V$ onto $S({V^\gamma}^*)\otimes
\wedge^{\bullet-l(\gamma)}V\otimes \wedge^{l(\gamma)} N^\gamma$.
Then the $\gamma$-component of their Gerstenhaber bracket is :
\begin{equation}
[b,\pi]_\gamma= {\textsf pr}_\gamma \circ \llbracket
b,\pi\rrbracket_\gamma.
\end{equation}
Moreover, the $\gamma$-component of the Gerstenhaber bracket of
$\pi$ with itself is obtained by
\begin{equation}
[\pi,\pi]_\gamma = \textsf{pr}_\gamma\circ\llbracket
\pi,\pi\rrbracket_\gamma.
\end{equation}

Let $\alpha$ and $\beta$ be two elements of $\Gamma$, $f_\alpha$ be
an element of $(S({V^\alpha}^*)\otimes
\wedge^{\bullet-l(\alpha)}V^\alpha \otimes
\wedge^{l(\alpha)}N^\alpha)^{Z(\alpha)}$ and $g_\beta$ be an element
of $(S({V^\beta}^*)\otimes \wedge^{\bullet-l(\beta)}V^\beta \otimes
\Lambda^{l(\beta)}N^\beta)^{Z(\beta)}$. If $l(\gamma)\ne
l(\alpha)+l(\beta)$, then the $\gamma$ component of $[f_\alpha,
g_\beta]$ vanishes. Suppose that any elements in $<\alpha>$ commutes
with any elements in $<\beta>$, where $<\alpha>, <\beta>$ are
subsets of $\Gamma$ of elements conjugate to $\alpha$ and $\beta$.
Then the Gerstenhaber bracket of $f_\alpha$ and $g_\beta$ is :
\begin{equation}
[f_\alpha,g_\beta]_{\gamma}=\sum_{{\tiny \begin{array}{c}\alpha'\in
<\alpha>, \beta'\in <\beta>, \alpha'\beta'=\gamma\\
l(\gamma)=l(\alpha'\beta')=l(\alpha)+l(\beta)\end{array}}}{\textsf
pr}_{\!\gamma}\circ\{f_{\alpha'},g_{\beta'}\},
\end{equation}
where $\{f_{\alpha'},g_{\beta'}\}$ is the usual Schouten-Nijenhuis
bracket of $f_{\alpha'}$ and $g_{\beta'}$.

The first and third author defined that a Poisson structure $\Pi$ on
$S(V^*)\rtimes \Gamma$ is a sum of elements like
\[
\pi_\gamma\in \Big(S({V^\gamma}^*)\otimes
\wedge^{2-l(\gamma)}V^\gamma \otimes
\wedge^{l(\gamma)}N^\gamma\Big)^{Z(\gamma)},\ \gamma\in C(\Gamma),
\]
with $l(\gamma)=0,2$ satisfying $[\Pi, \Pi]=0$.

\section{PBW property for constant and linear Poisson structures}
In this section, we prove that constant and linear Poisson
structures (with a mild assumption in linear cases) on
$S(V^*)\rtimes \Gamma$ can be quantized.

\subsection{Quantization of constant Poisson structures}

Following Equation (\ref{N-P-P-T}), a Poisson bracket $\pi$ on
$S(V^*)\rtimes \Gamma$ splits into a sum of $\pi_\gamma$,
\[
\pi_0+\sum_\gamma \pi_\gamma\in (S(V^*)\otimes \wedge ^2
V)^\Gamma\oplus \Big(\bigoplus_{\gamma\in \Gamma, l(\gamma)=2}
S({V^\gamma}^*)\otimes \wedge^2 N^\gamma\Big)^\Gamma.
\]
We say that $\pi_0+\sum_\gamma \pi_\gamma$ is a constant Poisson
structure if $\pi_0\in \wedge^2 V$ and $\pi_\gamma \in \wedge^2
N^\gamma$. We notice that in this constant case the
Braverman-Gaitsgory conditions \ref{BG} reduce to only one condition
(\ref{BG2}), which means that $\pi_\gamma$ has to be a cocycle. But
this is automatically satisfied as we know from Equation
(\ref{N-P-P-T}) that an element in $(S(V^*)\otimes \wedge ^2
V)^\Gamma\oplus \Big(\bigoplus_{\gamma\in \Gamma, l(\gamma)=2}
S({V^\gamma}^*)\otimes \wedge^2 N^\gamma\Big)^\Gamma$ is closed with
respect to the differential $b_\gamma$ and $b_0=0$.

\begin{theorem}
\label{thm:quant-constant} Any constant Poisson structure of
$S(V^*)\rtimes \Gamma$ is quantizable.
\end{theorem}

\begin{proof}
According to the above explanation, we know that any constant
Poisson structure satisfies the Braverman-Gaitsgory conditions
(\ref{BG1})-(\ref{BG3}). This implies that the quotient algebra
\[
H_\pi:= T(V^*)\rtimes \Gamma[[\hbar]]/\langle x\otimes y-y\otimes
x-\hbar(\pi_0(x,y)+\sum_{\gamma, l(\gamma)=2}
\pi_\gamma(x,y)\gamma)\rangle
\]
has PBW property, which defines a deformation quantization of the
algebra $S(V^*)\rtimes \Gamma[[\hbar]]$ with respect to the Poisson
structure $\pi=\pi_0+\sum_{\gamma\in \Gamma,
l(\gamma)=2}\pi_\gamma$.
\end{proof}

We remark that the PBW property of the algebra $H_\pi$ is checked in
Etingof-Ginzburg \cite{E-G}. Our proof is evident by using the
results from \cite{N-P-P-T}.
\subsection{Quantization of linear Poisson structures-abelian case}

In this subsection, we will assume that $\Gamma$ is an abelian group
which acts faithfully on $V$. According to representation theory of
a finite abelian group, $V$ is decomposed into a direct sum of 1 or
2 real dimensional subspaces where $\Gamma$ acts irreducibly.
$\gamma$ acts on 1 dimensional subspace with eigenvalue 1 or -1, and
on 2 dimensional subspace by rotation of finite order.

Let $\pi$ be a Poisson structure on $S(V^*)\rtimes \Gamma$, and
denote $\pi_0$ its identity component. The $\gamma$-components of
$\pi $ are null whenever $l(\gamma)\ne 0,2$ and take values in
$V^\gamma$. As we have remarked that if $l(\gamma)=1$, then $\gamma$
action on $N^\gamma$ has eigenvalue -1. There will not be any
nonzero element in ${V^\gamma}^*\otimes V^\gamma\otimes N^\gamma$
invariant under $\gamma$. Therefore, $\pi_\gamma$ is possible
nonzero only when $l(\gamma)=2$ or 0. And the eigenvalues of
$\gamma$ action on $N^\gamma$ are either -1 with multiplicity 2 or
roots of unity.

\begin{lemma}\label{lem:identity}
The identity component $\pi_0$ of $\pi$ defines a Lie bracket on
$V$.
\end{lemma}

\begin{proof}
It is a consequence of the fact that  $\pi_\gamma$ takes values in
$V^\gamma$ which is in the kernel of $\pi_{\gamma^{-1}}$. Therefore,
the condition that the Gerstenhaber bracket $[\pi,\pi]$ vanishes at
identity reduces to the Jacobi identity of $\pi_0$.
\end{proof}

\begin{lemma}\label{lem:bracket}
Let $\alpha$ and $\beta$ be elements of $\Gamma$ with
$l(\alpha)=l(\beta)=2$. Then, $\llbracket
\pi_\alpha,\pi_\beta\rrbracket =0$.
\end{lemma}

\begin{proof}
Let $x$ and $y$ be the coordinates on $N^\alpha$. It follows from
the $\Gamma$-invariance for Poisson structure that $\pi_\alpha$ has
to be $\beta$-invariant as $\alpha$ commutes with $\beta$ for any
$\beta\in \Gamma$:
\begin{equation}\label{eq:inv}
\pi_\alpha(\,^\beta\! x, \,^\beta\! y)=\,^\beta \!\pi_\alpha(x,y).
\end{equation}

We observe $\beta$ preserves $V^\alpha$ and $N^\alpha$ as $\beta$
commutes with $\alpha$. If $l(\beta)=2$, there are three
possibilities, 1) $N^\alpha\cap N^\beta=\{0\}$, 2)
$\dim(N^\alpha\cap N^\beta)=1$, 3) $N^\alpha=N^\beta$.

If $N^\alpha\cap N^\beta=\{0\}$, then $x,y$ are $\beta$ invariant
for $x,y$ in ${N^\alpha}^*$. Hence by Equation (\ref{eq:inv}),
$\pi_\alpha(x,y)=\,^\beta \pi_{\alpha}(x,y)$, which shows
$\pi_\alpha(x,y)$ is $\beta$ invariant. Similarly, we know that
$\pi_\beta$ takes value in $V^\alpha$. This shows that $\llbracket
\pi_\alpha, \pi_\beta\rrbracket=0$.

If $\dim(N^\alpha\cap N^\beta)=1$, then we know that both $\alpha$
and $\beta$ preserves $N^{\alpha, \beta}:=N^\alpha+N^\beta$ which is
of 3 dimension. Furthermore, we conclude that $\beta$'s ($\alpha$'s)
action on $N^\alpha$ (on $N^\beta$) has eigenvalue 1 and -1. Hence,
$N^{\alpha, \beta}$ is decomposed into a direct sum of $N_1\oplus
N_2\oplus N_3$ such that $\alpha$ acts on $N_1$ and $N_2$ by -1, and
 $N_3$ by 1, and $\beta$ acts on $N_1$ and $N_3$ by -1, and $N_2$ by
1. By Equation (\ref{eq:inv}), we know that $\pi_\alpha(^\beta x,\
^\beta y)=-\pi_\alpha(x,y)=\, ^\beta \pi_\alpha(x,y)$. This shows
that $\beta$ acts on $\pi_\alpha(x,y)$ by -1. Similarly $\alpha$
acts on the image of $\pi_\beta$ by -1. This shows that
$\pi_\alpha\in N_3^*\otimes N_1\otimes N_2$ and $\pi_\beta \in
N_2^*\otimes N_1\otimes N_3$ as $N_3^*$ (and $N_2^*$) is the only
1-dim subspace of ${V^{\alpha}}^*$ (of ${V^\beta}^*$) with a
nontrivial $\beta$ ($\alpha$) action. It is straightforward to check
that $\llbracket \pi_\alpha, \pi_\beta\rrbracket=0$ in $V^*\otimes
\wedge^3 V$.

If $N^\alpha=N^\beta$, $\beta$ acts on $N^\alpha=N^\beta$ with
determinant 1. This shows that $\pi_\alpha(^\beta x, ^\beta
y)=\pi_\alpha(x,y)$. By Equation (\ref{eq:inv}), we see that
$\pi_\alpha$ takes value in $V^\beta$, and similarly $\pi_\beta$
takes value in $V^\alpha$. Direct computation shows that $\llbracket
\pi_\alpha, \pi_\beta\rrbracket=0$.



\end{proof}

\begin{theorem}\label{thm:linear-abel}
Let $\Gamma$ be a finite abelian group which acts faithfully on a
finite dimensional vector space $V$. Then any linear Poisson
structure $\pi$ of $S(V^*)\rtimes\Gamma$ is quantizable.
\end{theorem}

\begin{proof}
Following the above lemmas, $\llbracket\pi,\pi\rrbracket_\gamma$ reduces to :
$$
\begin{array}{ccl}
\llbracket\pi,\pi\rrbracket_\gamma(x,y,z) &=& \pi_0(\pi_\gamma(x,y),z+^\gamma\!z)+\pi_0(\pi_\gamma(y,z),x+^\gamma\!x)+\pi_0(\pi_\gamma(z,x),y+^\gamma\!y) \\
&& +2\times
(\pi_\gamma(\pi_0(x,y),z)+\pi_\gamma(\pi_0(y,z),x)+\pi_\gamma(\pi_0(z,x),y)).
\end{array}
$$

This expression is zero whenever $l(\gamma)\ne2$. Suppose now that
$l(\gamma)=2$. If $x,y,z$ are all in $V^\gamma$, we get $0$ since
$V^\gamma$ is  the kernel of $\pi_\gamma$. If two of $x,y,z$ are in
$V^\gamma$, we also get $0$ for the same reason and because $\pi_0$
is $\gamma$-invariant. Suppose now that $x$ and $y$ are in
$N^\gamma$ and that $z$ is $\gamma$-invariant. Then,
$\llbracket\pi,\pi\rrbracket_\gamma(x,y,z)$ lies in $V^\gamma$, and
it is zero from the fact that $\pi$ is a Poisson bracket (Section
\ref{Gerstenhaber}): $[\pi,\pi]_\gamma=0$.
\end{proof}

\subsection{An important example}\label{sec:ex}

In order to prove the quantization theorem for general linear
Poisson structures, we consider an important example in this
subsection.

Let $\rho=\exp(\frac{2\pi i}{2n+1})$. Denote
$\alpha_k=\left(\begin{array}{ll}\rho^k&0\\
0&\rho^{-k}\end{array}\right)$, and $\beta_k=\left(\begin{array}{ll}0&\rho^{-k}\\
\rho^{k}&0\end{array}\right)$, and  $\Gamma_n=\{\alpha_k,
\beta_l:0\leq k,l\leq 2n \}$. $\Gamma_n$ is a finite group of order
$4n+2$ acting faithfully on $V=\mathbb{C}^2$, a complex 2-dim and
real 4-dim vector space. $\alpha_k$'s eigenvalues are $\rho^k$ and
$\rho^{-k}$, and $\beta_k$'s eigenvalues are $\pm 1$. Let $z_1,z_2$
be complex coordinate functions on $V$.

We consider linear Poisson structures on $S(V^*)\rtimes \Gamma_n$.
We first look at the Poisson structure on the identity component. As
$\alpha_k$ acts on $V$ diagonally, $\alpha_k$ acts on $V^*\otimes
\wedge^2 V$ also diagonally with eigenvalues $\rho^{3k}, \rho^{k},
\rho^{-k}$, and $\rho^{-3k}$. If $\rho^{3k}\ne 1$, there is no none
zero linear bivector field on $V$, which is $\alpha_k$ invariant.
Accordingly, if $\rho^3\ne 1$, there is no $\Gamma_n$-invariant
linear Poisson structure $\pi_0$ on $V$. If $\rho^3=1$, then $\pi_0$
is a linear combination of $z_1\bar{\partial}_1\wedge\partial_2$,
$z_2\partial_1\wedge\bar{\partial}_2$, and
$\bar{z}_1\partial_1\wedge \bar{\partial}_2$,
$\bar{z}_2\bar{\partial}_1\wedge
\partial_1$. If we assume that $\pi_0$ to be real, then we have
\[
\pi_0=az_1\bar{\partial}_1\wedge\partial_2+\bar{a}\bar{z}_1\partial_1\wedge
\bar{\partial}_2+bz_2\partial_1\wedge\bar{\partial}_2+\bar{b}\bar{z}_2\bar{\partial}_1\wedge
\partial_2.
\]
Furthermore by invariance with respect to the $\beta_k's$ action, we
have $a=-b$ and $\bar{a}=-\bar{b}$ in the above equation, i.e.
\begin{equation}\label{eq:pi-0}
\pi_0=a\left(z_1\bar{\partial}_1\wedge\partial_2-z_2\partial_1\wedge\bar{\partial}_2\right)+\bar{a}\left(\bar{z}_1\partial_1\wedge
\bar{\partial}_2-\bar{z}_2\bar{\partial}_1\wedge
\partial_2\right).
\end{equation}

Observe $\beta_k$ has real codimesion=2 fixed point subspace, while
$\alpha_k$ only fixes the origin of $V$. $V^k:=V^{\beta_k}$ is
determined by $\rho^{k}z_1-z_2=\rho^{-k}\bar{z}_1-\bar{z}_2=0$. The
normal subspace $N^k$ to $V^k$ is determined by
$\rho^kz_1+z_2=\rho^{-k}\bar{z}_1+\bar{z}_2=0$. Vector fields along
$N^k$ are spanned by $\rho^{-k}\partial_{1}-\partial_{2}$ and
$\rho^k\bar{\partial}_{1}-\bar{\partial}_{2}$. Therefore, the
Poisson structure at $\beta_k$ component can be written as
\[
\Pi_k=\left[c_k(\rho^kz_1+z_2)-\bar{c}_k(\rho^{-k}\bar{z}_1+\bar{z}_2)\right]
\left(\rho^{-k}\partial_{1}-\partial_{2}\right)\wedge
\left(\rho^k\bar{\partial}_{1}-\bar{\partial}_{2}\right).
\]
Furthermore, as $\alpha_l\beta_k\alpha_l^{-1}=\alpha_{k-2l}$, by
invariance of $\Pi_k$ with respect to the conjugation action of
$\Gamma_k$,
 $c_{2k}=c_0\rho^{-k}, 0\leq k\leq 2n $. Therefore, we
have
\begin{equation}\label{eq:pi-2k}
\begin{split}
\Pi_{2k}=&\left[c_0(\rho^kz_1+\rho^{-k}z_2)-\bar{c}_0(\rho^{-k}\bar{z}_1+\rho^k\bar{z}_2)\right]\\
&\left(\partial_{1}\wedge\bar{\partial}_{1}-\rho^{-2k}\partial_{1}\wedge
\bar{\partial}_{2}-\rho^{2k}\partial_{2}\wedge
\bar{\partial}_{1}+\partial_{2}\wedge
\bar{\partial}_{2}\right),\qquad 0\leq k\leq 2n.
\end{split}
\end{equation}
In summary, a Poisson structure $\Pi$ on $S(V^*)\rtimes \Gamma_n$ is
of the form, $\Pi=\pi_0+\sum_{k=0}^{2n}\Pi_{2k}$ where $\Pi_{2k}$ is
defined as in Equation (\ref{eq:pi-2k}), and $\pi_0$ vanishes unless
$n=3$. When $n=3$, $\pi_0$ is defined as in Equation (\ref{eq:pi-0})

By the same reason as in the proof of Theorem \ref{thm:linear-abel},
we conclude that if we assume $[\pi_0, \Pi_k]=0$, then
$\llbracket\pi_0, \Pi_k\rrbracket +\llbracket \Pi_k,
\pi_0\rrbracket=0$ for any $k$.

From Equation (\ref{BG2}), we see that as $\dim(N^\beta)=2$ the
$\llbracket\pi_\alpha, \pi_\beta\rrbracket(x,y,z)=0$ if $x,y,z$ are
all along the normal direction $N^\beta$. Furthermore, as
$\pi_\beta$ is a multiple of the highest wedge power of the normal
direction $N^\beta$, to have non-zero outcome two of the three
$x,y,z$ have to be from the normal direction $N^\beta$. This implies
that $\llbracket\pi_\alpha, \pi_\beta\rrbracket$ as an element in
$V^*\otimes \wedge^3 V$ is equal to the Schouten-Nijenhuis bracket
$[\pi_\alpha, \pi_\beta]$.

We use this observation to compute $\llbracket \Pi_{2k},
\Pi_{2l}\rrbracket$. A long but straightforward computation leads to
the following result at $\alpha_{2k-2l}$,
\[
\begin{split}
\llbracket\Pi_{2k},
\Pi_{2l}\rrbracket=&\left[c_0(\rho^k z_1+\rho^{-k}z_2)-\bar{c}_0(\rho^{-k}\bar{z}_1+\rho^k\bar{z}_2)\right]\\
         \times &\Big[(2\bar{c}_0\rho^l-\bar{c}_0\rho^{-l+2k}-\bar{c}_0\rho^{3l-2k})\partial_1\wedge
           \partial_2\wedge \bar{\partial}_1\\
         +&(-2\bar{c}_0\rho^{-l}+\bar{c}_0\rho^{l-2k}+\bar{c}_0\rho^{-3l+2k})
           \partial_1\wedge\partial_2\wedge\bar{\partial}_2\\
         +&(2c_0\rho^{-l}-c_0\rho^{l-2k}-c_0\rho^{-3l+2k})\bar{\partial}_1\wedge
           \partial_1\wedge \bar{\partial}_2\\
         +&(-2c_0\rho^l+c_0\rho^{-l+2k}+c_0\rho^{3l-2k})\bar{\partial}_1\wedge
           \partial_2\wedge \bar{\partial}_2\Big].
\end{split}
\]

Define for $0\leq k\leq 2n$,
\begin{equation}\label{eq:B}
B_{2k}:=(2n+1)(\rho^k-\rho^{-k})\left[-|c_0|^2\partial_2\wedge
\bar{\partial}_2+|c_0|^2\partial_1\wedge
\bar{\partial}_1+(\bar{c}_0)^2\partial_1\wedge\partial_2-c_0^2\bar{\partial}_1\wedge\bar{\partial}_2\right].
\end{equation}
With a long but straightforward computation, we are able to prove
\[
\sum_{p-q=2k, 0\leq q\leq 2n}\llbracket\Pi_{2p},
\Pi_{2q}\rrbracket=\partial^{\alpha_{2k}}B_{2k}.
\]
And it is not difficult to compute that
\begin{equation}\label{eq:ex-b}
\llbracket B_{2k}, \Pi_{2l}\rrbracket=0,\qquad 0\leq k,l\leq 2n;
\end{equation}
and
\[
\llbracket B_{2k}, \pi_0\rrbracket=0,\qquad 0\leq k \leq 2n,\ n=3.
\]
Therefore, we conclude with the following proposition
\begin{proposition}\label{prop:quant-ex}
For $\Gamma_n$ action on $V=\mathbb C^2$, any linear Poisson
structures on $S(V^*)\rtimes \Gamma_n$ can be quantized.
\end{proposition}
\begin{proof}
By the above computation, we see that the relation defining
\[
\begin{split}
H_\Pi:=& T(V)\rtimes \Gamma_n[[\hbar]]\\
       &\Big/\Big\langle
x\otimes y-y\otimes x-\hbar\Big(\pi_0(x,y)+\sum_{0\leq k\leq 2n}
\Pi_{2k}(x,y)\beta_{2k}\Big)-\hbar^2(\sum_{1\leq k\leq
2n}B_{2k}(x,y)\alpha_{2k})\Big\rangle
\end{split}
\]
satisfies the Braverman-Gaitsory conditions (\ref{BG1})-(\ref{BG3}).
This implies that $H_\Pi$ has PBW property, which shows that $H_\Pi$
is a deformation quantization of $S(V^*)\rtimes \Gamma_n$ along the
direction defined by $\Pi$.
\end{proof}

\begin{remark}We point out that in the proof of Proposition
\ref{prop:quant-ex}, there have to be nonzero terms $B_{2k}$ for
$0\leq k\leq 2n$ as $[\pi, \pi]$ is not zero. This is different from
the standard PBW theorem for Lie algebras where $B_{2k}$ can be
chosen to be zero. We will see in the following subsection that this
example is essentially the only case that $B_{2k}$ has to be
nonzero.
\end{remark}

\subsection{Quantization of linear Poisson structures-general case}

In this subsection $\Gamma$ is a finite group (not necessary
abelian) acting faithfully on a vector space $V$. We assume that $V$
is equipped with a $\Gamma$-invariant complex structure. We prove
the following theorem.

\begin{theorem}
\label{thm:quant-liniear}Let $\Gamma$ be a finite group acting
faithfully on a complex vector space $V$. Any real linear Poisson
structure on $S(V^*)\rtimes \Gamma$ is quantizable.
\end{theorem}

The proof of this theorem consists of several steps. We start with
recalling some results about finite subgroups of $GL(2, \mathbb C)$.
\begin{lemma}\label{lem:sl}A nonabelian finite subgroup $G$ of $SL(2,
\mathbb{C})$ must contain the element $\left(\begin{array}{cc}-1&0\\
0&-1\end{array}\right)\in SL(2, \mathbb{C})$.
\end{lemma}
\begin{proof}
We notice that the canonical action of $G$ on $\mathbb {C}^2$ is
irreducible. Otherwise, $G$ will be a subgroup of $GL(1,
\mathbb{C})\times GL(1, \mathbb{C})$, which is abelian. Therefore,
according to \cite{S}[Chapter 6, Proposition 17], the order of the
center of $G$ is divisible by 2. Therefore, there is an element in
$G$ of order 2. As $\left(\begin{array}{cc}-1&0\\
0&-1\end{array}\right)$ is the unique element in $SL(2, \mathbb{C})$
of order 2, we conclude that if $G$ is not abelian, $G$ contains $\left(\begin{array}{cc}-1&0\\
0&-1\end{array}\right)$.
\end{proof}

The following lemma is a corollary of \cite{D-a}[\S 26, Theorem 26].
\begin{lemma}
\label{lem:gl2}Let $\Gamma$ be a nonabelian finite subgroup of
$GL(2,\mathbb{C})$. If $\Gamma$ does not contain any matrix of the
form $\left(\begin{array}{cc}a&0\\
0&a\end{array}\right)$ for $a\ne 1$, then there is a natural number
$n$ such that $\Gamma$ is conjugate to the group $\Gamma_n$ as is
introduced in subsection \ref{sec:ex}.
\end{lemma}

\begin{proof}
We start by considering the intersection  $G=\Gamma\cap
SL(2,\mathbb{C})$. If $G$ is trivial, then there is an injective
group homomorphism
\[
\Gamma\to GL(2, \mathbb{C})/SL(2, \mathbb{C})\cong \mathbb{C}-\{0\}.
\]
This shows that $\Gamma$ is abelian, which contradicts the
assumption that $\Gamma$ is not abelian. Therefore, $G$ is a
nontrivial subgroup of $SL(2, \mathbb{C})$. The following discussion
is divided into two parts according to whether $G$ is abelian.
\begin{itemize}
\item $G$ is not abelian. Then by Lemma \ref{lem:sl}, $G$ contains
the element $\left(\begin{array}{cc}-1&0\\
0&-1\end{array}\right)$, which is in the center of $G$. This
contradicts to the assumption of this lemma.
\item $G$ is abelian. By conjugation with an invertible matrix, we can assume that $G$
contains a diagonal element like $A=\left(\begin{array}{cc}a&0\\
0&a^{-1}\end{array}\right)$ with $a\ne -1$. (We remark that any
element in $\Gamma$ is diagonalizable as $\Gamma$ is of finite
order.) Furthermore, we recall the fact that if
$B=\left(\begin{array}{ll}\alpha&\beta\\
\gamma&\delta\end{array}\right)$ commutes with $A$, then
$\beta=\gamma=0$. Hence, any element in $G$ is of the form $B=\left(\begin{array}{cc}\beta &0\\
0&\beta^{-1}\end{array}\right)$ for $\beta\in \mathbb{C}-\{0\}$.
Therefore, we conclude that $G$ is a cyclic subgroup of
$SL(2,\mathbb{R})$ isomorphic to $\{\left(\begin{array}{cc}\rho&0\\
0&\rho^{-1}\end{array}\right): \rho^{2n+1}=1\}$ for some $n\in
\mathbb{N}$. (If $\rho^{2n}=1$, then $\rho^n=-1$ and $G$ contains
the
element $\left(\begin{array}{cc}-1&0\\
0&-1\end{array}\right)$.)

We observe that if $B=\left(\begin{array}{cc}\alpha&\beta\\
\gamma&\delta\end{array}\right)\in GL(2, \mathbb{C})$ is a
normalizer of $G$, then $\alpha\beta=\gamma\delta=0$. Therefore, $B=\left(\begin{array}{ll}\alpha&0\\
0&\gamma\end{array}\right)$ or $\left(\begin{array}{ll}0&\beta\\
\delta&0\end{array}\right)$. This shows that any element in $\Gamma$
is either diagonal or of the form $\left(\begin{array}{ll}0&\beta\\
\delta&0\end{array}\right)$. As we have assumed that $\Gamma$ is not
an abelian group, there has to be a nonzero $B$ in $\Gamma$ of the form $\left(\begin{array}{ll}0&\beta\\
\delta&0\end{array}\right)$.

Compute $B^2=\left(\begin{array}{cc}\beta\delta&0\\
0&\delta\beta\end{array}\right)$. By the assumption of $\Gamma$,
$\beta\delta=1$. Therefore, $B=\left(\begin{array}{cc}0&a\\
a^{-1}&0\end{array}\right)\in \Gamma$. Now choose $U=\left(\begin{array}{cc}0&a^{\frac{1}{2}}\\
a^{-\frac{1}{2}}&0\end{array}\right)$, and consider the group
$\tilde{\Gamma}=U^{-1}\Gamma U$ which is again not abelian and does
not contain any matrix of the form $\left(\begin{array}{cc}a&0\\
0&a\end{array}\right)$. Under this isomorphism, we see that
$\tilde{G}=G=\{\left(\begin{array}{cc}\rho &0\\
0&\rho^{-1}\end{array}\right): \rho^{2n+1}=0\}$ and $\tilde{\Gamma}$
contains a matrix
$\beta_0=U^{-1}BU=\left(\begin{array}{ll}0&1\\
1&0\end{array}\right)\in \tilde{\Gamma}$.

Now if there is any other element $C$ in $\tilde{\Gamma}$ of the form $\left(\begin{array}{ll}0&\beta'\\
\delta'&0\end{array}\right)$ then by the same arguments as $B$, we
know that $\beta'=1/\delta'$. Furthermore, as $\beta_0B=\left(\begin{array}{ll}\delta'&0\\
0&\delta'^{-1}\end{array}\right)\in \tilde{G}=G$. This implies that
$B=\left(\begin{array}{cc}0&\rho\\
\rho^{-1}&0\end{array}\right)$ with $\rho^{2n+1}=1$.

Next if there is any element $D=\left(\begin{array}{cc}\alpha&0\\
0&\gamma\end{array}\right)$ in $\tilde{\Gamma}$,
compute $\beta_0D\beta_0D=\left(\begin{array}{cc}\alpha\gamma &0\\
0&\alpha\gamma\end{array}\right)$. As $\tilde{\Gamma}$ has no
element
like $\left(\begin{array}{cc}a&0\\
0&a\end{array}\right)$ with $a\ne 1$, we conclude that
$\alpha=1/\gamma$ and $D$ belongs to $G$.

Summarizing the above analysis, we have seen that
there exists $n\in \mathbb{N}$, such that $\tilde{\Gamma}=\{\left(\begin{array}{cc}0&\rho^i\\
\rho^{-i}&0\end{array}\right), \left(\begin{array}{cc}\rho^i&0\\
0&\rho^{-i}\end{array}\right): 0\leq i\leq 2n, \rho^{2n+1}=1\}$.
\end{itemize}
\end{proof}

\noindent{\it Proof of theorem \ref{thm:quant-liniear}}:

With Theorem \ref{thm:linear-abel}, it is sufficient to work with
$\Gamma$ which is nonabelian. We choose a $\Gamma$-invariant
hermitian metric on $V$ which always exists as $\Gamma$ is finite.
We prove that there exists a choice for $B_\gamma$ such that the
Braverman-Gaitsory conditions (\ref{BG1})-(\ref{BG3}) are
satisfied.\\

\noindent{\bf Step I:} The Braverman-Gaitsory condition (\ref{BG1})
is satisfied automatically by the assumption on $\pi_\alpha$.\\

\noindent{\bf Step II:} In the following, we make a proper choice
for $B_\gamma$ with $l(\gamma)=4$ such that the Braverman-Gaitsory
condition (\ref{BG2}) is satisfied.

As $\Gamma$ is acting on a complex vector space, the fixed subspace
of any group element $\gamma$ is of even real codimension. Let
$\pi_0$ be the linear Poisson structure at the identity component,
$\pi_\alpha$ be the linear Poisson structure at the $\alpha$
component with $l(\alpha)=2$.

We look at $\llbracket\pi_0, \pi_\alpha\rrbracket$, $\llbracket
\pi_\alpha, \pi_0 \rrbracket$, and $\llbracket \pi_\alpha,
\pi_\beta\rrbracket$ with $l(\alpha)=l(\beta)=2$. By the same
arguments as in the proof of Theorem \ref{thm:linear-abel}, we have
that for any $\alpha$ with $l(\alpha)=2$, $\llbracket\pi_0,
\pi_\alpha\rrbracket+\llbracket \pi_\alpha, \pi_0 \rrbracket=0$ as
$\pi_0+\sum_{\alpha}\pi_\alpha$ is a Poisson structure. Therefore,
we are reduced to look at $\llbracket \pi_\alpha,
\pi_\beta\rrbracket$.

We observe that if $V^\alpha=V^\beta$, then $\wedge^2
N^\alpha=\wedge^2 N^\beta$ vanishes on functions depending only on
variables in $V^\alpha=V^\beta$. It is easy to compute that
$\llbracket \pi_\alpha, \pi_\beta \rrbracket=0$. This reduces us to
the situation that $V^\alpha\ne V^\beta$.

When $V^\alpha\ne V^\beta$, we have the equation
$V^\alpha+V^\beta=V$ as both $V^\alpha$ and $V^\beta$ are complex
subspaces of $V$ of complex codimension 1. Furthermore the equality
$V^\alpha+V^\beta=V$ implies that $V^{\alpha\beta}=V^\alpha\cap
V^\beta$ which is of complex codimension 2 and real codimension 4.
We notice that in this case both $\alpha$ and $\beta$ fix every
point in $V^{\alpha\beta}$, and also preserve the normal direction
$N^{\alpha\beta}$. We are interested in $\llbracket \pi_\alpha,
\pi_\beta\rrbracket$, which is now at $\alpha\beta$ component. As
$\llbracket \pi_\alpha, \pi_\beta \rrbracket$ is only a tri-vector
field  supported at $N^{\alpha\beta}$ with $l(\alpha\beta)=4$, we
know that $\llbracket \pi_\alpha, \pi_\beta \rrbracket$ is 0 in the
Hochschild cohomology of $S(V)\rtimes \Gamma$.

Now we fix an element $\gamma$ with $l(\gamma)=4$, then we know from
the previous paragraph that if $\llbracket \pi_\alpha,
\pi_\beta\rrbracket $ is nonzero at $\gamma=\alpha\beta$, then
$\alpha$ and $\beta$ acts on $V$ preserving $N^{\gamma}$ and fixing
every element in $V^\gamma$. This leads us to look at the subgroup
$\Gamma_\gamma$ of $\Gamma$ whose elements act trivially on
$V^\gamma$. $\Gamma_\gamma$ contains all $\alpha$ such that
$V^\gamma\subset V^\alpha$, which acts on $N^\gamma$ faithfully.

Let $\alpha\in \Gamma_\gamma$ with $l(\alpha)=2$. By the assumption
on $\pi_\alpha$, it is an element in ${V^\alpha}^* \otimes \wedge^2
N^\alpha$, which can be written as a sum of two terms
$\pi_\alpha^1+\pi_\alpha^2$ as ${V^\alpha}^*={V^{\gamma}}^*\oplus
({N^\gamma}^*\cap {V^\alpha}^*)$. We easily see that $\pi_\alpha^1$
will not contribute to $\llbracket \pi_\alpha, \pi_\beta\rrbracket$
as $N^\alpha$ and $N^\beta$ are orthogonal to $V^\gamma$. This shows
that to study $\llbracket \pi_\alpha, \pi_\beta\rrbracket$, it is
enough to assume that $\pi_\alpha$ belongs to ${N^\gamma}^*\otimes
\wedge^2 N^\alpha$. Furthermore, if $n_\alpha$ is the holomorphic
vector along $N^\alpha$ and $v_\alpha$ is the holomorphic vector
along ${V^\alpha}^*$ in $N^\gamma$, then we can write
$\pi_\alpha=(c_\alpha v_\alpha-\bar{c}_\alpha \bar{v}_\alpha)
n_\alpha\wedge \bar{n}_\alpha$ for some complex number $c_\alpha$ by
the fact that $\pi_\alpha$ is real.

If $\Gamma_\gamma$ is abelian, then by Lemma \ref{lem:bracket}, we
know that $\llbracket \pi_\alpha, \pi_\beta\rrbracket=0$ for any
$\alpha,\beta$, and therefore we set $B_{\delta}=0$ for $\delta\in
\Gamma_\gamma$. In the following, we assume that $\Gamma_\gamma$ is
not abelian. If $\Gamma_\gamma$ contains an element $\nu$ which acts
on
$N^\gamma$ of the form $\left(\begin{array}{cc}a&0\\
0&a\end{array}\right)$ with a unitary number $a\ne 1$, then it is
easy to check $\nu_*(\pi_\alpha)=(a^{-1}c_\alpha
v_\alpha-\bar{a}^{-1}\bar{c}_\alpha \bar{v}_\alpha) n_\alpha\wedge
\bar{n}_\alpha$ is not invariant under $\nu$ unless $c_\alpha=0$.
This shows that $\pi_\alpha$ has to be zero if it is invariant under
$\nu$ and therefore $\llbracket \pi_\alpha, \pi_\beta\rrbracket=0$
in this case. We choose $B_\gamma=0$ for this type of $\gamma$.
Therefore, for nonzero $B_\gamma$, we only need to consider the
situation that
$\Gamma_\gamma$ is not abelian and contains no element of the form $\left(\begin{array}{cc}a&0\\
0&a\end{array}\right)$ with $a\ne 1$. By Lemma \ref{lem:gl2},
$\Gamma_\gamma$ action on $N^\gamma$ is isomorphic to the situation
studied in subsection \ref{sec:ex}. And we can choose $B_\gamma$ as
Equation (\ref{eq:B}).\\

\noindent{\bf Step III:} we prove that $\sum_{\alpha, \gamma,
l(\alpha)\leq2, l(\gamma)=4}\llbracket B_\gamma, \pi_\alpha
\rrbracket=0$ with the choices of $B_\gamma$ introduced in Step II.

We decompose the above sum into 2 parts
\begin{enumerate}
\item $\sum_{\alpha, \gamma,
l(\alpha)\leq 2, l(\gamma)=4, l(\gamma\alpha)\geq 4}\llbracket
B_\gamma, \pi_\alpha \rrbracket$;
\item $\sum_{\alpha, \gamma, l(\alpha)\leq 2, l(\gamma)=4,
l(\gamma\alpha)= 2}\llbracket B_\gamma, \pi_\alpha\rrbracket$.
\end{enumerate}

We consider the part of sum with $l(\gamma\alpha)\geq 4$. Let
$\delta=\gamma\alpha$. Define $D_\delta=\sum_{\gamma\alpha=\delta,
l(\alpha)\leq 2, l(\gamma)=4}\llbracket B_\gamma,
\pi_\alpha\rrbracket$, which is an element in $\wedge^3V$. We notice
that $D_\delta$ is $\partial^\delta$ closed as
\[
\begin{split}
\partial(\sum_{\alpha,
\gamma}\llbracket B_\gamma, \pi_\alpha\rrbracket)=&\sum_{\alpha,
\gamma}\llbracket \partial B_\gamma, \pi_\alpha\rrbracket\\
=&\sum_{\alpha, \gamma=\beta\lambda}\llbracket \llbracket \pi_\beta,
\pi_\lambda\rrbracket, \pi_\alpha\rrbracket\\
=&0.
\end{split}
\]
According to Equation (\ref{N-P-P-T}), we see that $D_\gamma$ in
$\wedge ^3 V$ is a zero cocycle as $l(\gamma)\geq 4$. On the other
hand, we see that $D_\gamma$ is in $\wedge ^3 V$.  If we define
elements in $V$ of degree -1 and elements in $V^*$ of degree $1$,
then $D_\gamma$ is of degree -3, and $\partial^\gamma$ is of degree
0. We see that $S(V^*)\otimes \wedge^2 V$ has degree greater or
equal to $-2$. Therefore, $\partial^\gamma (S(V^*)\otimes \wedge
^2V)$ has no term with degree less than  -2 since $\deg(\partial
^\gamma)=0$. This shows that if $D_\gamma\ne0$, it cannot be a
coboundary of $\partial^\gamma$. This shows that $D_\gamma$ has to
be zero as $D_\gamma$ is a trivial cocycle by Equation
(\ref{N-P-P-T}).

We consider the part of sum with $l(\gamma\alpha)\leq 2$. This
implies that $V^\gamma\subset V^\alpha$ because otherwise
$V^\gamma+V^\alpha=V$ and $V^{\gamma\alpha}=V^\gamma\cap V^\alpha$
which is of real codimension 6. In this case, we know that
$\gamma\alpha$ is also in $\Gamma_\gamma$. Therefore, we can use
Equation (\ref{eq:ex-b}) to conclude that $\llbracket B_\gamma,
\pi_\alpha\rrbracket=0$.\\

In conclusion, Steps I-III show that Bravermam-Gaitsgory conditions
(\ref{BG1})-(\ref{BG3}) are satisfied for any linear Poisson
structures on $S(V^*)\rtimes \Gamma$ and a proper choice of
$B_\gamma$. Therefore, by PBW property, we see that the algebra
\[
\begin{split}
H_\Pi:=& T(V)\rtimes \Gamma[[\hbar]]\\
       &\Big/\Big\langle
x\otimes y-y\otimes x-\hbar(\pi_0(x,y)+\sum_{\alpha}
\pi_{\alpha}(x,y) \alpha)-\hbar^2(\sum_{\gamma}B_{\gamma}(x,y)
\gamma)\Big\rangle
\end{split}
\]
defines a deformation quantization of $S(V^*)\rtimes\Gamma$ along
the direction of $\pi_0+\sum_\alpha \pi_\alpha$. \qquad $\Box$
\section{Hochschild cohomology and non commutative Poisson cohomology}
In this section, we would like to study various properties and
examples of the Poisson structures and algebras we constructed in
the previous section. To state our results, we fix some convention.
Note that if $\alpha$ and $\beta$ are conjugate to each other inside
$\Gamma$, then $l(\alpha)=l(\beta)$. Therefore, it is legitimate to
define {\bf codimension} of a conjugacy class of $\Gamma$ by the
codimension of an element $\alpha\in \Gamma$. Define $c_k$ to be the
number of conjugacy class of $\Gamma$ with codimension $k$. In the
following, we use $\mathbb{R}((\hbar))$ to stand for the algebra of
Laurent Polynomials of $\hbar$. For any algebra $A$ over the ring
$\mathbb R[[\hbar]]$, we use $A((\hbar))$ to stand for the extension
$A\otimes_{\mathbb R[[\hbar]]}\mathbb R((\hbar))$.

\begin{proposition}\label{prop:per-cyclic}
The periodic cyclic homology of the algebra $H_\pi$ for any constant
or linear Poisson structures on $S(V^*)\rtimes \Gamma$ is equal to
\[
\begin{split}
HP_0(H_\pi((\hbar)))&=\sum_{k}\mathbb R((\hbar))^{\times c_k}
=\mathbb{R}((\hbar))^{\times |C(\Gamma)|}\\
HP_1(H_\pi((\hbar)))&=0.
\end{split}
\]
\end{proposition}
\begin{proof}
By Getzler-Goodwillie \cite{G-W}, we know that as the periodic
cyclic homology is invariant under deformation,
$HP_\bullet(H_\pi((\hbar)))$ is equal to $HP_\bullet(S(V^*)\rtimes
C(\Gamma))((\hbar))$. By the computation in \cite{B-N},
$HP_\bullet(S(V^*)\rtimes \Gamma)$ is equal as is stated above.
\end{proof}

\subsection{Hochschild cohomology of symplectic reflection algebra}
In this subsection, we restrict ourselves to quantization of a
special type of constant Poisson structures. These are called
symplectic reflection algebras. Let $V$ be a symplectic vector space
with standard symplectic 2-form $\omega$. Let $\pi$ be the
associated Poisson structure of $\omega$. For $\gamma$ with
$l(\gamma)=2$, consider the restriction $\omega_\gamma$ of $\omega$
to $N^\gamma$. $\omega_\gamma$ is an invertible bilinear operation
on $N^\gamma$. Define $\pi_\gamma$ be the inverse of
$\omega_\gamma$. We choose $c_\gamma\in \mathbb R$ for all $\gamma$
with $l(\gamma)=2$ with $c_{\alpha\gamma\alpha^{-1}}=c_{\gamma}$ for
all $\alpha, \gamma\in \Gamma$. Define a constant Poisson structure
$\Pi$ on $S(V^*)\rtimes \Gamma$ by
\[
\Pi=\pi+\sum_{\gamma, l(\gamma)=2}c_\gamma \pi_\gamma.
\]
By Thm \ref{thm:quant-constant}, $S(V^*)\rtimes \Gamma$ has a
deformation quantization with respect to $\Pi$. This algebra can be
written as
\[
H_{\omega, c}=T(V^*)\otimes \Gamma[[\hbar]]/\langle x\otimes
y-y\otimes x-\hbar(\pi(x,y)+\sum_{\gamma, l(\gamma)=2}c_\gamma
\pi_\gamma(x,y)\gamma)\rangle.
\]
This algebra is called symplectic reflection algebra by Etingof and
Ginzburg \cite{E-G}.

To compute the Hochschild cohomology of $H_{\omega, c}$, we compute
the Poisson cohomology of $\Pi$ first. We recall that the Poisson
cohomology of a Poisson structure $\Pi$ on an algebra $A$ is defined
to be the cohomology of the complex $H^\bullet(A, A)$ with the
differential $\partial(a)=[\Pi, a]$, where $[\ ,\ ]$ is the
Gerstenhaber bracket on $H^\bullet(A,A)$.
\begin{proposition}
\label{prop:Poisson-constant}The Poisson cohomology
$H_\Pi^\bullet(S(V^*)\rtimes \Gamma)$ is equal to
\[
H_\Pi^\bullet(S(V^*)\rtimes \Gamma)=\mathbb R^{\times c_{\bullet}}.
\]
\end{proposition}

\begin{proof}
We introduce a grading on
\[
H^\bullet(S(V^*)\rtimes \Gamma, S(V^*)\rtimes
\Gamma)=\Big(\bigoplus_{\gamma\in \Gamma}S({V^\gamma}^*)\otimes
\wedge^{\bullet-l(\gamma)}V^\gamma \otimes
\wedge^{l(\gamma)}N^\gamma\Big)^\Gamma
\]
by setting elements in $S({V^\gamma}^*)\otimes
\wedge^{\bullet-l(\gamma)}V^\gamma \otimes
\wedge^{l(\gamma)}N^\gamma$ of degree $l(\gamma)$.

By Poisson cohomology with respect to $\Pi$, we mean the cohomology
on $H^\bullet(S(V^*)\rtimes \Gamma, S(V^*)\rtimes \Gamma)$ with the
differential $d_\Pi$ defined by taking the generalized
Schouten-Nijenhuis bracket $d_\Pi f=[\Pi,f ]$ for $f\in
H^\bullet(S(V^*)\rtimes \Gamma, S(V^*)\rtimes \Gamma)$. According to
\cite{H-T}[Thm. 3.4], the Poisson differential $d_\Pi$ is compatible
with the above filtration with respect to $l(\gamma)$ as if
$\deg(f)=i$, $(d_\Pi f)_\gamma=0$ if $l(\gamma)\ne i$ or $i+2$.
Therefore, we can use spectral sequence associated to the filtration
defined by the grading $l$ to compute the Poisson cohomology
$H_\Pi$.

The $E_0$ of the spectral sequence associated to the above
filtration is the Poisson cohomology with respect to the Poisson
structure $\pi$ which is the component of $\Pi$ supported at
identity of the group $\Gamma$ on the graded complex
$\text{Gr}(H^\bullet(S(V^*)\rtimes \Gamma, S(V^*)\rtimes \Gamma))$,
i.e.
\[
\text{Gr}^p(H^\bullet(S(V^*)\rtimes \Gamma, S(V^*)\rtimes
\Gamma))=\Big(\bigoplus_{\gamma, l(\gamma)=p}S({V^\gamma}^*)\otimes
\wedge^{\bullet-l(\gamma)}V^\gamma \otimes
\wedge^{l(\gamma)}N^\gamma\Big)^\Gamma.
\]
As $\omega$ is a symplectic form, the Poisson cohomology of $\pi$
can be computed easily
\[
E_{1}^{p,q}=\left\{\begin{array}{lc}H_\pi^{p}(S(V^*)\rtimes \Gamma)=\mathbb R^{\times c_{p}},&\qquad q=0,\\
\{0\},&\qquad q\ne 0.\end{array}\right.
\]

We notice that the codimension of any group element $\gamma$ is even
because $\gamma$ preserves the symplectic structure. Therefore,
$E_1^{p,q}= 0$ if and only if $p+q$ is odd. This implies that the
spectral sequence degenerates at $E_1$. Therefore we have
\[
H_\Pi^\bullet(S(V^*)\rtimes \Gamma)=H_\pi^\bullet(S(V^*)\rtimes
\Gamma)=\mathbb R^{\times c_{\bullet}}.
\]
\end{proof}
\begin{theorem}\label{thm:coh-constant}The Hochschild cohomology of
$H_{\omega,c}((\hbar))$ is computed as follows,
\[
H^\bullet(H_{\omega,c}((\hbar)), H_{\omega, c}((\hbar)))=\mathbb
R((\hbar))^{\times c_{\bullet}}.
\]
\end{theorem}

\begin{proof}
The proof of this theorem uses the spectral sequence with respect to
the $\hbar$-filtration. The $E_0$ terms are the Poisson cohomology.
We notice that $H^\bullet_\Pi(S(V^*)\rtimes \Gamma)$ is trivial when
$\bullet$ is odd, and conclude that the spectral sequence
degenerated at $E_1$. Therefore,  the Hochschild (co)homology equals
the Poisson (co)homology.
\end{proof}

\begin{remark}
Theorem \ref{thm:coh-constant} is a generalization of
\cite{E-G}[Theorem 1.8 (i)]. Here, with more information about the
generalized Schouten-Nijenhuis bracket, we are able to avoid the
restriction of \cite{E-G} on ``except possibly a countable set".
\end{remark}

In Thm \ref{thm:coh-constant}, we see that for a constant Poisson
structure $\Pi$, if the Poisson structure at the identity component
$\pi_0$ is the inverse of a symplectic structure, then its Poisson
cohomology and Hochschild cohomology are determined by $\pi_0$
completely.

In the following we show one example that if the Poisson structure
$\pi_0$ is degenerated, then the Poisson cohomology of $\Pi$ depends
also on the information of $\Pi$ at other conjugacy classes.

Consider $(V=\mathbb {R}^2, \omega=dx\wedge dy)$, and $\mathbb
Z_2:=\mathbb{Z}/2\mathbb {Z}=\{ 1, e\}$ acts on $\mathbb {R}^2$ by
$e(x,y)=(-x,-y)$. Denote $\pi=-\partial_x\wedge \partial_y$. Observe
that 0 is the only fixed point of $e$. For any constant $c$,
consider $\Pi=c\pi e: x\wedge y\mapsto c\pi(x,y)e$ which defines a
constant Poisson structure on $S(V^*)\rtimes \mathbb{Z}_2$. In the
following, we show that the Poisson cohomology of $\Pi$ does
distinguish $\Pi$ from the trivial Poisson structure and also those
Poisson structures in Proposition \ref{prop:Poisson-constant}.

\begin{proposition}\label{prop:counter-constant}
\[
\begin{split}
H_{\Pi}^0&=(S(V^*))^{\mathbb Z_2}\\
H_{\Pi}^1&=\{f_1\partial_x+f_2\partial_y\in
(S(V^*)\otimes V)^{\mathbb Z_2}:\partial_xf_1(0)+\partial_yf_2(0)=0\}\\
H_{\Pi}^2&=(S(V^*)\otimes \wedge^2V)^{\mathbb Z_2}.
\end{split}
\]
\end{proposition}
\begin{proof}
The computation for $H^0$ are trivial as there is only degree=2
elements supported at $e$ by Equation (\ref{N-P-P-T}).

For $H^1$, we first observe that the image of $d^{\Pi}$ on
$(S(V^*))^{\mathbb Z_2}$ is trivial by Equation (\ref{N-P-P-T}). For
the kernel of $d^{\Pi}$, we compute
\[
[f_1\partial_x+f_2\partial_y,
\Pi]=c(\partial_x(f_1)+\partial_y(f_2))\partial_x\wedge
\partial_y|_{0},
\]
where $c(\partial_x(f_1)+\partial_y(f_2))\partial_x\wedge
\partial_y|_{0}$ is the restriction of $c(\partial_x(f_1)+\partial_y(f_2))\partial_x\wedge
\partial_y$ to the origin $0$. Therefore, $H_{\Pi}^1=\{f_1\partial_x+f_2\partial_y\in
(S(V^*)\otimes V)^{\mathbb
Z_2}:\partial_xf_1(0)+\partial_yf_2(0)=0\}$.

For $H_2^{\Pi}$, we notice that $d^{\Pi\\}$ vanishes as there is no
higher degree terms. For the image of $d^{\Pi}$, from the previous
computation, we see that
$Im(d^{\Pi})=\mathbb{R}\partial_x\wedge\partial_y|_{0}$. Therefore
we conclude from Equation (\ref{N-P-P-T}) that
$H_{\Pi}^2=(S(V^*)\otimes \wedge^2V)^{\mathbb Z_2}$.
\end{proof}

\subsection{Hochschild cohomology of $\mathcal U(\mathfrak g)\rtimes \Gamma$}
Let $V=\mathfrak g$ be a Lie algebra such that its bracket is
$\Gamma$ invariant. The Lie bracket on $\mathfrak g$ defines a
Poisson structure on $V^*$, which also defines a Poisson structure
on $S(V)\rtimes \Gamma$. Then, $\mathcal U(\mathfrak g)\rtimes
\Gamma$ is a quantization of $\pi$ on $S(V)\rtimes \Gamma$. We
notice that the Poisson bracket $\pi$ does not have any other
$\gamma$-component for $\gamma$ different from the unity, i.e.
$\pi=\pi_0$.

Now, $\mathcal U(\mathfrak g)\rtimes \Gamma$ is a filtered Koszul
algebra over $\mathbb R\Gamma$. Therefore, it has a Koszul
resolution and a small complex which calculates its Hochschild
cohomology. This complex splits into a direct sum of subcomplexes:
$$
(CK^\bullet(\mathcal U(\mathfrak g)\rtimes \Gamma),\partial) =
\bigoplus_{\gamma\in C(\Gamma)}(CK_\gamma^\bullet,\partial_\gamma)
$$
where $CK_\gamma^\bullet=(\mathcal U(\mathfrak g)\otimes
\wedge^\bullet \mathfrak g^*)^{Z(\gamma)}$ and :
$$
\begin{array}{ccl}
\partial_\gamma(f)(x_0,\cdots,x_n)&=& \displaystyle\sum_{i=0}^n (-1)^i f(x_0, \cdots  \widehat{x}_i\cdots, x_n) (x_i-^\gamma x_i)\\[5mm]
&& \hspace{-5cm}+\displaystyle\sum_{i=0}^n (-1)^{i}\big[x_i,f(x_0,
\cdots  \widehat{x}_i\cdots, x_n)\big] +
\displaystyle\sum_{i<j}(-1)^{j-i-1} f(x_{0},\cdots
x_{i-1},[x_{i},x_{j}],  x_{i+1}, \cdots \widehat{x}_j\cdots, x_{n}).
\end{array}
$$

The brackets in the above formula stand for the Lie bracket of
$\mathfrak g$ and for the action of $\mathfrak g$ on $\mathcal
U(\mathfrak g)$. Notice that the PBW property of $\mathcal
U(\mathfrak g)$ implies that the symmetrization map from $S\mathfrak
g$ to $\mathcal U(\mathfrak g)$ is an isomorphism of $\mathfrak
g$-modules as well as $\Gamma$-modules.

According to \ref{Gerstenhaber}, the Poisson complex splits as well
in a direct sum of subcomplexes
$$
(C^\bullet_{\pi}(\mathcal U(\mathfrak g)\rtimes
\Gamma),\partial^\pi) = \bigoplus_{\gamma\in
C(\Gamma)}(C_\gamma^\bullet,\partial^\pi _\gamma),
$$
where $C_\gamma^\bullet=(S(\mathfrak g^\gamma)\otimes
\wedge^{\bullet-l(\gamma)}{\mathfrak g^\gamma}^* \otimes
\wedge^{l(\gamma)}{N^\gamma}^*)^{Z(\gamma)}$ with the differential :
$$
\begin{array}{ccl}
\partial^\pi _\gamma(f)(x_0,\cdots,x_{n-l(\gamma)},y_1,\cdots, y_{l(\gamma)})
&=&\\
&&\hspace{-5cm}\displaystyle\sum_{0\leq i\leq n-l(\gamma)} (-1)^{i}\big[x_i,f(x_0, \cdots  \widehat{x}_i\cdots, x_{n-l(\gamma)},y_{1},\cdots,y_{l(\gamma)})\big] \\[5mm]
&& \hspace{-5cm}+ \displaystyle\sum_{i<j\leq
n-l(\gamma)}(-1)^{j-i-1} f(x_{0},\cdots x_{i-1},[x_{i},x_{j}],
x_{i+1}, \cdots \widehat{x}_j\cdots,
x_{n-l(\gamma)},y_1,\cdots,y_{l(\gamma)})
\\[5mm]
&& \hspace{-5cm}+ \displaystyle\sum_{i\leq
n-l(\gamma),j}(-1)^{n-l(\gamma)-i+j-1} f(x_{0},\cdots
\widehat{x}_i\cdots,
x_{n-l(\gamma)},y_1,\cdots,[x_i,y_j],\cdots,y_{l(\gamma)})
\end{array}
$$
where the $x$-variables belong to $\mathfrak g^\gamma$ and the $y$'s
belong to $N^\gamma$. Notice that, as the bracket of $\mathfrak g$
is $\Gamma$-invariant, $N^\gamma$ is a $\mathfrak
g^\gamma$-submodule of $ \mathfrak g$. This defines the last summand
of $\partial^\pi_\gamma$.

We define a map $\psi$ from the Poisson complex $C_\gamma^\bullet$
to the Koszul complex $CK_\gamma^\bullet$. For any $f$ in
$C_\gamma^n$, we put :
$$
\psi(f) := \Lambda^n\mathfrak g  \longrightarrow
\Lambda^{n-l(\gamma)}\mathfrak g^\gamma \otimes
\Lambda^{l(\gamma)}N^\gamma  \mathop{\longrightarrow}\limits^f
S\mathfrak g^\gamma \mathop{\longrightarrow}\limits^{Sym} \mathcal
U(\mathfrak g^\gamma)  \longrightarrow \mathcal U(\mathfrak g)
$$
where the first and last map are the usual projection and injection,
and $Sym$ is the classical symmetrization map from $S\mathfrak
g^\gamma$ to $\mathcal U(\mathfrak g^\gamma)$.

\begin{lemma}
The map $\psi ^*$ is a morphism of complexes.
\end{lemma}

\begin{proof}
We have to check that $\psi$ commutes with the differentials. For
this purpose, let us decompose $\partial_\gamma$ into a sum of three
terms, $\partial^1_\gamma +\partial^2_\gamma +\partial^3_\gamma$
corresponding to the three components of its definition. We also use
the following decomposition of $\Lambda^{n+1}\mathfrak g$ coming
from the direct sum $\mathfrak g=\mathfrak g^\gamma\oplus N^\gamma$
:
$$
\Lambda^{n+1}\mathfrak g=\bigoplus_{p=0}^{l(\gamma)}\Lambda^{n+1-p}\mathfrak g^\gamma\otimes \Lambda^pN^\gamma
$$

First, as $\psi(f)$ needs $l(\gamma)$ independent variables in
$N^\gamma$, it follows easily that $\partial^1_\gamma (\psi(f))=0$.
For the same reason, we check that  $\partial(\psi(f))$ is null on
the $\Lambda^{n+1-p} \mathfrak g^\gamma \otimes \Lambda^p N^\gamma$
whenever $p<l(\gamma)$. Therefore, $\partial(\psi(f))$ reduces to a
map from $\Lambda^{n+1-l(\gamma)}\mathfrak g^\gamma\otimes
\Lambda^{l(\gamma)}N^\gamma$ to $\mathcal U(\mathfrak g^\gamma)$.

To complete the proof, we check that the two formulas of the
differentials agree on $\Lambda^{n+1-l(\gamma)}\mathfrak
g^\gamma\otimes \Lambda^{l(\gamma)}N^\gamma$ thanks to the fact that
the symmetrization map $Sym$ is a $\mathfrak g$-morphism.
\end{proof}

\begin{theorem}\label{thm:coh-ug}
The map $\psi$ is a quasi-isomorphism. Therefore the Hochschild
cohomology of $\mathcal U(\mathfrak g)\rtimes \Gamma$ is isomorphic,
as a graded vector space, to the Poisson cohomology of $S(\mathfrak
g)\rtimes \Gamma$ with the Poisson bracket induced by the Lie
bracket of $\mathfrak g$ :
$$
H^\bullet( \mathcal U(\mathfrak g)\rtimes \Gamma, \mathcal
U(\mathfrak g)\rtimes \Gamma)\simeq H_\pi^\bullet(S(\mathfrak
g)\rtimes \Gamma).
$$
\end{theorem}

\begin{proof}
Let us introduce a formal parameter $\hbar$ and work over the formal
power series $\mathbb R[[\hbar]]$. Introduce the $\mathbb
R[[\hbar]]$-Lie algebra $\mathfrak g_\hbar$ whose underlying
$\mathbb R[[\hbar]]$-module is $\mathfrak g[[\hbar]]$ with the Lie
bracket given by :
$$
[x,y]_\hbar = \hbar[x,y],
$$
for $x$ and $y$ in $\mathfrak g$. Then, the Koszul complex
$CK^\bullet(\mathcal U(\mathfrak g_h) \rtimes \Gamma)$ specializes
to that of $S(\mathfrak g) \rtimes \Gamma$ for $\hbar=0$, and, for
$\hbar=1$ to that of $\mathcal U(\mathfrak g) \rtimes \Gamma$.

Similarly, the Poisson complex $C_\pi^\bullet(\mathcal U(\mathfrak
g_\hbar)\rtimes \Gamma)$ has a zero differential for $\hbar=0$ and
specializes to $C_\pi^\bullet(\mathcal U(\mathfrak g) \rtimes
\Gamma)$ for $\hbar=1$.

The map $\psi$ extends to the $\mathbb R[[\hbar]]$-context, and
defines a morphism of $\mathbb R[[\hbar]]$-complexes. It follows
from Subsection \ref{Koszul} that $\psi$ specializes to a
quasi-isomorphism for $\hbar=0$. The result is then a consequence of
the following standard lemma,  which can be found in \cite{G-H}.
\end{proof}

\begin{lemma}
Let $C_1^\bullet[[\hbar]]$ and $C_2^\bullet[[\hbar]]$ be two
topologically free $\mathbb R[[\hbar]]$-complexes, and $\psi$ a
morphism of $\mathbb R[[\hbar]]$-complexes. Suppose that $\psi$
specializes to a quasi-isomorphism for $\hbar=0$. Then $\psi$ is a
quasi-isomorphism.
\end{lemma}

We remark that in this paper we have always worked with real
coefficient $\mathbb R$. Theorem \ref{thm:coh-ug} is true for a
general field $\mathbb K$ with characteristic 0.

\subsection{Examples of linear Poisson structures}
In this section, we provide a large class of linear Poisson
structures coming from invariant Lie algebra structures.

We assume that $\mathfrak g$ be a Lie algebra and $\Gamma$ be a
finite group acting on $\mathfrak g$ preserving its Lie bracket.
We choose a $\Gamma$ invariant metric on $\mathfrak g$. Let $V$ be
the dual of $\mathfrak g$ with the linear Poisson structure $\pi$
from the Lie bracket. Accordingly,  $\Gamma$ acts on $V$ preserving
the Poisson structure $\pi$.

For any $\gamma\in \Gamma$ with $l(\gamma)=2$, let $N^\gamma$ be the
subspace of $V$ normal to $V^\gamma$. As $V=V^\gamma\oplus N^\gamma$
and $V^*={V^\gamma}^*\oplus {N^\gamma}^*$. One can decompose
$\pi=V^*\otimes \wedge ^2 V =({V^\gamma}^*\oplus
{N^\gamma}^*)\otimes (N^\gamma\wedge N^\gamma\oplus N^\gamma\otimes
V^\gamma \oplus V^\gamma\otimes N^\gamma\oplus V^\gamma\wedge
V^\gamma)$. We define $\pi_\gamma$ to be the projection of $\pi$
onto the component $(V^\gamma)^*\otimes \wedge^2 N^\gamma$.

\begin{proposition}\label{prop:ex-linear}The collection $\Pi=\pi+\sum_{\gamma\in l(\gamma)=2}c_\gamma
\pi_\gamma$ with constant $c_\gamma$ satisfying
$c_{\alpha\gamma\alpha^{-1}}=c_\gamma$ for any $\alpha\in \Gamma$
defines a linear Poisson structure on $S(V^*)\rtimes \Gamma$.
\end{proposition}
\begin{proof}
As is explained in subsection \ref{Koszul}, there is no $Z(\gamma)$
invariant section in $S({V^\gamma}^*)\otimes \wedge ^\bullet
V^\gamma\otimes \wedge ^{l(\gamma)}N^\gamma$ with odd $l(\gamma)$.
Furthermore, by Equation (\ref{N-P-P-T}), any tri-vector field on a
$\gamma$-component with $l(\gamma)=4$ is a trivial cocycle. This
implies that
\[
[\pi_\alpha, \pi_\beta]=0,
\]
if neither $\alpha$ nor $\beta$ is the identity of $\Gamma$.

To prove that $\Pi$ is a Poisson structure, it is sufficient to
prove
\[
[\pi, \pi_\gamma]=0,\ \gamma\in \Gamma.
\]
When $\gamma$ is identity, the above equation is from the fact that
$\pi$ is the Lie Poisson structure.

When $\gamma\ne 0$, we can decompose $\pi$ according to the
decomposition $V=V^\gamma \oplus N^\gamma$, $V^*={V^\gamma}^*\oplus
{N^\gamma}^*$. As $\pi\in V^*\otimes V\wedge V $ is $\gamma$
invariant,  $\pi$ is a sum of the following terms
\[
\begin{split}
\pi_{111}\in {V^\gamma}^*\otimes V^\gamma\wedge
V^\gamma,\qquad&\pi_{122}\in {V^\gamma}^*\otimes N^\gamma\wedge
N^\gamma,\\
\pi_{221}\in {N^\gamma}^*\otimes N^\gamma\wedge V^\gamma,\qquad
&\pi_{222}\in {N^\gamma}^*\otimes N^\gamma\wedge N^\gamma.
\end{split}
\]

We compute
\[
\begin{split}
[\pi_{111}, \pi_{111}]\in {V^\gamma}^*\otimes V^\gamma\wedge
V^\gamma\wedge V^\gamma,\quad& [\pi_{111}, \pi_{122}]\in
{V^\gamma}^*\otimes
V^\gamma\wedge N^\gamma \wedge N^\gamma, \\
[\pi_{111}, \pi_{221}]\in {N^\gamma}^*\otimes N^\gamma\wedge
V^\gamma\wedge
V^\gamma,\quad& [\pi_{111}, \pi_{222}]=0,\\
[\pi_{122}, \pi_{122}]=0,\quad& [\pi_{122}, \pi_{222}]\in
{V^\gamma}^*\otimes N^\gamma\wedge N^\gamma\wedge N^\gamma,\\
[\pi_{122}, \pi_{221}]\in (V^\gamma)^*\otimes N^\gamma\wedge
N^\gamma \wedge V^\gamma\oplus &{N^\gamma}^*\otimes N^\gamma\wedge
N^\gamma\wedge N^\gamma,\\
[\pi_{221}, \pi_{221}]\in {N^\gamma}^*\otimes N^\gamma\wedge
V^\gamma\wedge V^\gamma,\quad& [\pi_{221}, \pi_{222}]\in
{N^\gamma}^*\otimes N^\gamma \wedge N^\gamma \wedge V^\gamma,\\
[\pi_{222}, \pi_{222}]\in {N^\gamma}^*\otimes N^\gamma\wedge
N^\gamma\wedge N^\gamma& .
\end{split}
\]
According to the fact that $[\pi, \pi]=0$ and $N^\gamma\wedge
N^\gamma\wedge N^\gamma=0$ for $l(\gamma)=2$, we see
\[
\begin{split}
[\pi_{111}, \pi_{111}]&=0,\\
[\pi_{111},\pi_{221}]+[\pi_{221}, \pi_{221}]&=0,\\
[\pi_{111}+\pi_{221}, \pi_{122}]&=0,\\
[\pi_{122}, \pi_{222}]&=0,\\
[\pi_{122}, \pi_{122}]&=0,\\
[\pi_{221}, \pi_{221}]&=0.
\end{split}
\]

In particular, this implies $[\pi, \pi_{122}]=0$. Noticing that
$\pi_\gamma=\pi_{122}$, we conclude that the Schouten bracket of
$[\pi, \pi_\gamma]=0$.
\end{proof}

In the following, we construct explicit examples of linear Poisson
structures using Proposition \ref{prop:ex-linear} on $\mathbb{R}^3$
with the $\mathbb{Z}/2\mathbb{Z}=\{1, e\}$ action by $e(x,y,z)=(-x,
-y, z)$. In this case, the fixed point subspace of $e$ is
$\{x=y=0\}$, which is of codimension 2.

\begin{ex}\label{ex:linear-1}
Denote $\pi_1=z\partial_x\wedge\partial_y$. One can easily check
that $[\pi_1, \pi_1]=0$. And the Poisson structure constructed from
Proposition \ref{prop:ex-linear} is $\Pi_1=\pi_1+\pi_1|_{x=y=0}
e=z\partial_x\wedge \partial_y+z\partial_x\wedge
\partial_y|_{x=y=0}e$.
\end{ex}

\begin{ex}\label{ex:linear-2}
Denote $\pi_2=z\partial_x\wedge \partial_y+x\partial_x\wedge
\partial_z-y\partial_y\wedge \partial_z$. Noticing that $[\partial_x\wedge \partial_y,
x\partial_x-y\partial_y]=0$, we have $[\pi_2, \pi_2]=0$. And the
Poisson structure constructed from Proposition \ref{prop:ex-linear}
is $\Pi_2=\pi_2+\pi_2|_{x=y=0}e=z\partial_x\wedge
\partial_y+x\partial_x\wedge\partial_z-y\partial_y\wedge \partial_z+z\partial_x\wedge
\partial_y|_e$.
\end{ex}

We compute the 0-th Poisson cohomology of $\Pi_1$ and $\Pi_2$ to
distinguish them.
\begin{enumerate}
\item If $f\in (S\mathbb{R}^3)^{\mathbb {Z}/2\mathbb {Z}}$ is a 0-th Poisson cocycle with respect to $\Pi_1$, i.e.
$\Pi_1(f)=0$, then $[\pi_1, f]=0$ i.e.
$\partial_x(f)=\partial_y(f)=0$, which is also a sufficient
condition. Therefore, $H_{\Pi_1}^0=\{f\in (S\mathbb{R}^3)^{\mathbb
{Z}/2\mathbb {Z}}:
\partial_x(f)=\partial_y(f)=0\}$, which means $f$ is a polynomial
depending only on $z$.
\item If $f\in (S\mathbb{R}^3)^{\mathbb {Z}/2\mathbb {Z}}$ is a 0-th
Poisson cocycle with respect to $\Pi_2$, i.e. $\Pi_2(f)=0$, then
$\pi_2(f)=0$ and $\pi_1(f)|_{x=y=0}=0$. This is equivalent to
$zf_x+yf_z=0, zf_y+xf_z=0, xf_x-yf_y=0$, and $
f_x|_{x=y=0}=f_y|_{x=y=0}=0$. If we write
$f=\sum_{k,m,n}c_{kmn}x^ky^mz^n$, we have that
\[
\begin{split}
(k+1)c_{k+1mn-1}&+(n+1)c_{km-1n+1}=0\\
(m+1)c_{km+1n-1}&+(n+1)c_{k-1mn+1}=0\\
kc_{kmn}-mc_{kmn}&=0\\
c_{10n}=c_{01n}&=0.
\end{split}
\]
From the above third equation,  $c_{kmn}=0$ if $k\ne m$. From the
above first and second equation,
$(k+1)c_{k+1k+1n}+(n+2)c_{kkn+2}=0$. From this we can quickly
conclude that $f$ is a polynomial of $xy-1/2 z^2$, and $H^0_{\Pi_2}$
consists of polynomials on $xy-1/2 z^2$.
\end{enumerate}

We notice that in the above two examples the 0-th Poisson cohomology
of $\Pi_1$ (and $\Pi_2$) is isomorphic to the 0-th Poisson
cohomology of the restriction $\pi_1$ (and $\pi_2$) of $\Pi_1$ (and
$\Pi_2$) to the identity component of $\mathbb Z_2$. In the
following proposition, we prove that this is a general phenomena.
\begin{proposition}
\label{prop:poisson-coh}If the group $\Gamma$ is abelian, the
Poisson cohomology of $\Pi$ is isomorphic to the Poisson cohomology
of $\pi$.
\end{proposition}

\begin{proof}We construct a quasi-isomorphism
\[
\begin{split}
\Psi: \Big( (\bigoplus_{\gamma\in \Gamma} S({V^\gamma}^*)\otimes
\wedge^{\bullet-l(\gamma)}V^\gamma \otimes
\wedge^{l(\gamma)}N^\gamma)^\Gamma,\ & [\pi,\ ] \Big)
\longrightarrow
\\
&\Big( (\bigoplus_{\gamma\in \Gamma} S({V^\gamma}^*)\otimes
\wedge^{\bullet-l(\gamma)}V^\gamma \otimes
\wedge^{l(\gamma)}N^\gamma)^\Gamma,\ [\Pi,\ ] \Big).
\end{split}
\]

Given $X=\sum_\gamma X_\gamma\in (\bigoplus_{\gamma\in \Gamma}
S({V^\gamma}^*)\otimes \wedge^{\bullet-l(\gamma)}V^\gamma \otimes
\wedge^{l(\gamma)}N^\gamma)^\Gamma$, we define $\Psi(X)$ as a sum of
$\sum_\gamma \Psi(X_\gamma)$. Define $\Psi(X_\gamma)$ as a sum of
$1/k!\psi_{\gamma, \alpha_1, \cdots,\alpha_k}(X_\gamma)$ for $k=0,
1, 2,\cdots$. If $l(\gamma\alpha_1\cdots\alpha_k)\ne l(\gamma)+2k$,
then define $\psi_{\gamma, \alpha_1, \cdots, \alpha_k}(X_\gamma)=0$;
if $l(\gamma\alpha_1\cdots\alpha_k)=l(\gamma)+2k$, then we define
$\psi_{\gamma, \alpha_1, \cdots, \alpha_k}(X_\gamma)$ to be the
projection of $X_\gamma$ down to the component
\[
\Big(S({V^{\gamma\alpha_1\cdots\alpha_k}}^*)\otimes\wedge^\bullet
V^{\gamma\alpha_1\cdots\alpha_k} \otimes \wedge
^{l(\gamma)}N^\gamma\otimes \wedge^2 N^{\alpha_1}\otimes \cdots
\otimes \wedge^2 N^{\alpha_k}\Big)^\Gamma.
\]
We notice that since $l(\gamma\alpha_1\cdots\alpha_k)\leq \dim(V)$,
there are only a finite number of occasions that $\psi_{\gamma,
\alpha_1, \cdots, \alpha_k}$ is not zero. Therefore, the map
$\Psi(X_\gamma)$ is well-defined.

To prove that $\Psi$ is a chain map, we prove the following equation
for $\psi_{\gamma, \alpha_1, \cdots, \alpha_k}$ in
$H^\bullet(S(V^*)\rtimes \Gamma, S(V^*)\rtimes \Gamma)$ at component
$\gamma\alpha_1\cdots\alpha_k$,
\begin{equation}\label{eq:psi}
\sum_{\alpha_1\cdots \alpha_k=\delta}\psi_{\gamma, \alpha_1, \cdots,
\alpha_k}([\pi,
X_\gamma])=\sum_{\alpha_1\cdots\alpha_k=\delta}[\pi_{\alpha_k},
k\psi_{\gamma, \alpha_1, \cdots,\alpha_{k-1}}(X_\gamma)]+[\pi,
\psi_{\gamma, \alpha_1, \cdots, \alpha_{k-1}, \alpha_k}(X_\gamma)].
\end{equation}

Equation (\ref{eq:psi}) can be proved by induction. When $k=0$, the
identity is trivial. Assume that Equation (\ref{eq:psi}) is true for
$k=n$. We look at the case when $k=n+1$. Using the induction
assumption, we have
\[
\begin{split}
\sum_{\alpha_1\cdots\alpha_{n+1}=\delta}&\psi_{\gamma, \alpha_1,
\cdots, \alpha_{n+1}}([\pi,
X_\gamma])=\sum_{\alpha_1\cdots\alpha_{n+1}=\delta}\psi_{\gamma,
\alpha_1, \cdots, \alpha_{n+1}}(\psi_{\gamma, \alpha_1, \cdots,
\alpha_{n}}([\pi,
X_\gamma]))\\
=&\sum_{\alpha_1\cdots\alpha_{n+1}=\delta}\psi_{\gamma, \alpha_1,
\cdots, \alpha_{n+1}}\Big(n[\pi_{\alpha_n}, \psi_{\gamma, \alpha_1,
\cdots,\alpha_{n-1}}(X_\gamma)]+[\pi, \psi_{\gamma, \alpha_1,
\cdots, \alpha_n}(X_\gamma)]\Big).
\end{split}
\]

Looking at the contribution of $[\pi, \psi_{\gamma, \alpha_1,
\cdots, \alpha_n}(X_\gamma)]$ at $\gamma\alpha_1\cdots\alpha_{n+1}$,
we decompose $\psi_{\gamma, \alpha_1, \cdots, \alpha_n}(X_\gamma)$
and $\pi$ according to $V=V^{\alpha_{n+1}}\oplus N^{\alpha_{n+1}}$.
As both $\pi$ and $\psi_{\gamma, \alpha_1, \cdots,
\alpha_n}(X_\gamma)$ are $\alpha_{n+1}$ invariant, we can write
$\pi=\pi_0+\pi_1+\pi_2$ and $\psi_{\gamma, \alpha_1, \cdots,
\alpha_n}(X_\gamma)=X_0+X_1+X_2$ such that $\pi_i, X_i$ contain $i$
number of vector fields along $N^{\alpha_{n+1}}$. It is easy to see
that $[\pi_0+\pi_1, X_0+X_1]$ will contribute zero at
$\gamma\alpha_1\cdots\alpha_{n+1}$ by the facts that $\pi$ is a
linear Poisson structure and $[\pi_0+\pi_1, X_0+X_1]$ at
$\gamma\alpha_1\cdots\alpha_{n+1}$ is the restriction to
$V^{\gamma\alpha_1\cdots\alpha_{n+1}}$ of the component
\[
S({V^{\gamma\alpha_1\cdots\alpha_{n+1}}}^*)\otimes\wedge^\bullet
V^{\gamma\alpha_1\cdots\alpha_{n+1}} \otimes \wedge
^{l(\gamma)}N^\gamma\otimes \wedge^2 N^{\alpha_1}\otimes \cdots
\otimes \wedge^2 N^{\alpha_{n+1}}
\]
in $[\pi_0+\pi_1, X_0+X_1]$. Furthermore, $[\pi_2, X_2]$ vanishes as
it is a 3 vector field normal to $V^{\alpha_{n+1}}$, which is of
codimension 2. Therefore,  $[\pi, \psi_{\gamma, \alpha_1, \cdots,
\alpha_n}(X_\gamma)]$ at $\gamma\alpha_1\cdots\alpha_{n+1}$ is equal
to
\[
[\pi, X_2]+[\pi_2, X]=[\pi, \psi_{\gamma, \alpha_1, \cdots,
\alpha_{n+1}}(X_\gamma)]+[\pi_{\alpha_{n+1}}, \psi_{\gamma,
\alpha_1, \cdots, \alpha_{n}}(X_\gamma)].
\]
Next, we observe that $\psi_{\gamma, \alpha_1, \cdots,
\alpha_{n+1}}([\pi_{\alpha_n}, \psi_{\gamma, \alpha_1,
\cdots,\alpha_{n-1}}(X_\gamma)])$ is equal to $[\pi_{\alpha_n},
\psi_{\gamma, \alpha_1, \cdots,\alpha_{n-1},
\alpha_{n+1}}(X_\gamma)]$ as the space $N^{\alpha_n}$ and
$N^{\alpha_{n+1}}$ are orthogonal to each other by the assumption
that $l(\gamma)+2n+2=l(\gamma\alpha_1\cdots\alpha_{n+1})$. Hence, we
have the above Equation (\ref{eq:psi}) for $k=n+1$.

Using Equation (\ref{eq:psi}), we can easily check that $\Psi$ is a
chain map. To prove that $\Psi$ is a quasi-isomorphism, we look at
the filtration with respect to the grading $l(\gamma)$ as is used in
the proof of Proposition \ref{prop:Poisson-constant}. It is straight
forward to see that the induced chain map of $\Psi$ is identity at
$E_1$ as for $k=0$, $\psi_\gamma(X_\gamma)=X_\gamma$, which implies
that $\Psi$ is a quasi-isomorphism.
\end{proof}
\section{Examples of quantization}

In this section, we study some examples of quantization of
constant Poisson structures. We look at
$\mathbb{Z}_n=\mathbb{Z}/n\mathbb{Z}$ action on $\mathbb{R}^2$ by
the rotation
\[
\gamma: (x,y)\mapsto (\cos(\frac{2\pi}{n})x-\sin(\frac{2\pi}{n})y,
\sin(\frac{2\pi}{n})x+\cos(\frac{2\pi}{n})y),\ \gamma^n=1,
\]
where $\gamma$ is the generator of $\mathbb{Z}/n\mathbb{Z}$. If we
introduce complex coordinates $z=x+iy, \bar{z}=x-iy$, then the
above action is diagonalized
\[
\gamma:(z,\bar{z})\mapsto (\exp(\frac{2\pi i}{n})z,
\exp(-\frac{2\pi i}{n})\bar{z}).
\]

We study the Poisson structure of the following form
$\pi:\wedge^2\mathbb R^2 \to \mathbb R \mathbb{Z}_n$ by
$\pi(x,y)=\gamma$. In complex coordinates
$\pi(z,\bar{z})=-i/2\gamma$. By \cite{H-T}[Corollary 4.2], $\pi$
defines a noncommutative structure on $S(\mathbb{R}^2)\rtimes
\mathbb{Z}_n$, and by Proposition \ref{prop:Poisson-constant}, $\pi$
can be quantized. In the following two sections, we study properties
of quantization of $\pi$.

\subsection{A Moyal type formula}
In this subsection, we provide an explicit formula for quantization
of $\pi$, which is a generalization of Moyal product. We would like
to point out that many of the following formulas appear already in
Nadaud's paper \cite{N}. We prove that in the case of finite group,
this product is convergent. We start with introducing several
operators on $S(\mathbb{R}^2)$. We work with complex coordinates
$z=x+iy,\ \bar{z}=x-iy$.

Define $D_{z}, D_{\bar{y}}:S(\mathbb{R}^2)\to S(\mathbb{R}^2)$ as
\begin{eqnarray*}
D_{z}(f)=\frac{f(z,\bar{z})-f(e^{\frac{2\pi
i}{n}}z,\bar{z})}{(1-e^{\frac{2\pi i}{n}})z}\ \ \ \ \ &
D_{\bar{z}}(f)=\frac{f(z,\bar{z})-f(z,e^{-\frac{2\pi
i}{n}}\bar{z})}{(1-e^{-\frac{2\pi i}{n}})\bar{z}}.
\end{eqnarray*}

Define $\sigma_z, \sigma_{\bar{z}}:S(\mathbb{R}^2)\to
S(\mathbb{R}^2)$ as
\begin{eqnarray*}
\sigma_z(f)(z,\bar{z})=f(e^{\frac{2\pi i}{n}}z,\bar{z})\ \ \ \
&\sigma_{\bar{z}}(f)(z,\bar{z})=f(z,e^{-\frac{2\pi i}{n}}\bar{z}).
\end{eqnarray*}

Let $q=\exp(2\pi i/n)$. Define $[k]_q=1+q+\cdots +q^{k-1}$. Define
the following star product $\star$ on $S(\mathbb{R}^2)\rtimes
\mathbb{Z}_n$.

Define $f\gamma^k\star g\gamma^l$ by
\[
f\gamma^k\star g\gamma^l=\sum_{j=0}^\infty
\frac{(\frac{i\hbar}{2})^j}{[j]_q!} (D_z)^j(f)(\sigma_z
D_{z})^j(\gamma^k(g))\gamma^{j+k+l}
\]

To prove the associativity of $\star$, we study properties of
$D_z,D_{\bar{z}}$.
\begin{lemme}\label{lem:D-x}
\[
D_z^k(fg)=\sum_{i=0}^k \frac{[k]_q!}{[k-i]_q! [i]_q!}D_z^i(f)
\sigma_z^i D_{\bar{z}}^{k-i}(g),\ \ \ \ k\geq0.
\]
\end{lemme}
\begin{proof}
We prove this by induction. When $k=0$, this identity is trivial.

Assume that this identity holds for $k$. For $k+1$, we compute
\begin{eqnarray*}
D_z^{k+1}(fg)&=&D_z(\sum_{i=0}^k \frac{[k]_q!}{[k-i]_q!
[i]_q!}D_z^i(f) \sigma_z^i D_z^{k-i}(g))\\
&=&\sum_{i=0}^k\frac{[k]_q!}{[k-i]_q! [i]_q!}
D_z(D_z^i(f)\sigma_z^iD_z^{k-i}(g))\\
&=&\sum_{i=0}^k(\frac{[k]_q!}{[k-i]_q!
[i]_q!}D_z^{i+1}(f)\sigma_z^{i+1}D_z^{k-i}(g)+\frac{[k]_q!}{[k-i]_q!
[i]_q!}D_z^i(f)D_z
\sigma_z^i D_z^{k-i}(g))\\
&=&D_z^{k+1}(f)\sigma^{k+1}_z(g)+fD_z^{k+1}(g)\\
&+&\sum_{i=1}^k\Big(\frac{[k]_q!}{[k-i+1]_q!
[i-1]_q!}D_z^{i}(f)\sigma_z^iD_z^{k-i+1}(g)+\frac{[k]_q!}{[k-i]_q!
[i]_q!}q^i
D_z^i(f)\sigma_z^{i} D_z^{k-i+1}(g)\Big)\\
&=&D_z^{k+1}(f)\sigma^{k+1}_z(g)+\sum_{i=1}^k\frac{[k+1]_q!}{[k-i+1]_q!
[i]_q!} D_z^i(f)\sigma_z^{i} D_z^{k-i+1}(g)+fD_z^{k+1}(g)\\
&=&\sum_{i=0}^{k+1}\frac{[k+1]_q!}{[k-i+1]_q! [i]_q!}D_z^i(f)
\sigma_z^i D_z^{k+1-i}(g)).
\end{eqnarray*}
\end{proof}
We start to prove the associativity of $\star$.
\begin{eqnarray*}
&&(f\star g)\star h\\
&=&\sum_k(\frac{i\hbar}{2})^{k}\frac{1}{[k]_q!}D_z^k(f)\sigma^k_zD^{k}_{\bar{z}}(g)\gamma^{k}\star
h\\
&=&\sum_{k,l}(\frac{i\hbar}{2})^{k+l}\frac{1}{[l]_q![k]_q!}D_z^l\Big(D_z^k(f)\sigma^k_zD^{k}_{\bar{z}}(g)\Big)
\sigma_z^lD_{\bar{z}}^l(\gamma^{k}(h))\gamma^{k+l}
\end{eqnarray*}
Using Lemma \ref{lem:D-x}, we have that the above product is equal
to
\begin{eqnarray*}
&=&\sum_{k,l}(\frac{i\hbar}{2})^{k+l}\frac{1}{[l]_q![k]_q!}\frac{[l]_q!}{[l-m]_q![m]_q!}D_z^{m+k}(f)\sigma^m_z
D_z^{l-m}\sigma^k_z
D^k_{\bar{z}}(g)\sigma_z^lD_{\bar{z}}^l(\gamma^{k}(h))\gamma^{k+l}\\
&=&\sum_{s,t,k}(\frac{i\hbar}{2})^{s+t+l}\frac{1}{[k]_q![s]_q![t]_q!}D_z^{s+k}(f)q^{kt}\sigma^{s+k}_zD_z^tD^k_{\bar{z}}(g)
\sigma_z^{s+t}D_{\bar{z}}^{s+t}(\gamma^{k}(h))\gamma^{k+s+t}\\
&=&\sum_{s,t,k}(\frac{i\hbar}{2})^{s+t+l}\frac{1}{[k]_q![s]_q![t]_q!}D_z^{s+k}(f)q^{kt}\sigma^{s+k}_zD_z^tD^k_{\bar{z}}(g)
\sigma_{\bar{z}}^{s+t}D_{\bar{z}}^{s+t}\sigma_z^{k}\sigma_{\bar{z}}^k(h))\gamma^{k+s+t}\\
&=&\sum_{s,t,k}(\frac{i\hbar}{2})^{s+t+l}\frac{1}{[k]_q![s]_q![t]_q!}D_z^{s+k}(f)q^{kt}\sigma_z^{s+k}D_z^tD^k_{\bar{z}}(g)
\sigma_z^{s+t+k}\sigma^k_{\bar{z}}q^{-k(s+t)}D_{\bar{z}}^{s+t}(h)\gamma^{k+s+t}\\
\end{eqnarray*}
On the other hand, we compute
\begin{eqnarray*}
&&f\star(g\star h)\\
&=&\sum_k (\frac{i\hbar}{2})^k \frac{1}{[k]_q!}f\star
(D_z^k(g)\sigma_z^kD_{\bar{z}}^k(h)\gamma^{k}\\
&=&\sum_{k,l} (\frac{i\hbar}{2})^{k+l}
\frac{1}{[l]_q![k]_q!}D_z^l(f)\sigma_z^lD_{\bar{z}}^l\Big(D_z^k(g)\sigma_z^kD_{\bar{z}}^k(h)\Big)
\gamma^{k+l}
\end{eqnarray*}
Applying the similar formula for $D_z$ as Lemma \ref{lem:D-x}, we
have
\begin{eqnarray*}
&=&\sum_{k,s,t}(\frac{i\hbar}{2})^{k+s+t}\frac{1}{[k]_q![s+t]_q!}\frac{[s+t]_{q^{-1}}!}{[s]_{q^{-1}}![t]_{q^{-1}}!}
D_z^{s+t}(f)\sigma_z^{s+t}D_{\bar{z}}^sD_z^k(g)\sigma_z^{s+t}\sigma_{\bar{z}}^sD_{\bar{z}}^{t}\sigma_z^kD_{\bar{z}}^k(h)\gamma^{k+s+t}\\
&=&\sum_{k,s,t}(\frac{i\hbar}{2})^{k+s+t}\frac{1}{[k]_q![s+t]_q!}\frac{[s+t]_{q^{-1}}!}{[s]_{q^{-1}}![t]_{q^{-1}}!}
D_z^{s+t}(f)\sigma_z^{s+t}D_{\bar{z}}^sD_z^k(g)\sigma_z^{s+t+k}\sigma_{\bar{z}}^sD_{\bar{z}}^{t+k}(h)\gamma^{k+s+t}
\end{eqnarray*}
In the above equation, we make the change $s\mapsto k, t\mapsto s,
k\mapsto t$, then we have
\[
\sum_{k,s,t}(\frac{i\hbar}{2})^{k+s+t}\frac{1}{[t]_q![k+s]_q!}\frac{[k+s]_{q^{-1}}!}{[k]_{q^{-1}}![s]_{q^{-1}}!}
D_z^{k+s}(f)\sigma_z^{k+s}D_{\bar{z}}^kD_z^t(g)\sigma_z^{s+t+k}\sigma_{\bar{z}}^kD_{\bar{z}}^{s+t}(h)\gamma^{k+s+t}
\]
which is identified with the above expression of $(f\star g)\star
h$. We remark that $[k]_q=[k]_{q^{-1}}q^{k-1}$ and
$[k]_q!=[k]_{q^{-1}}!q^{(k-1)k/2}$. Therefore, we conclude that
$\star$ defines an associative deformation of $S(\mathbb R^2)\rtimes
\mathbb Z_n$, whose $\hbar$ component is equal to $i/2D_z\otimes
\sigma_zD_{\bar{z}}$, which is cohomologous to the Poisson structure
$\pi$.

We remark that our proof of associativity of $\star$ is slightly
different from \cite{N}. One can view $f\star g$ as an extension
of the Moyal product as follows
\[
f\star g=m\circ \exp_q(\frac{i\hbar}{2}(D_z\otimes \sigma_z
D_{\bar{z}}\otimes \gamma))(f\otimes g\otimes 1).
\]
Nadaud proved the associativity of $\star$ analogous to the
associativity of Moyal product by using the property of the
$q$-exponential. Here our proof is more straightforward, and it
leads to the following more precise formula for $\star$.

We have the following property for the operator $D_z$.
\begin{proposition}\label{prop:D-x}
\[
D_z^m(f)=\frac{\sum_{i=0}^m
(-1)^{m-i}\frac{[m]_q!}{[m-i]_q![i]_q!}q^{i(i-1)/2}\sigma_z^{m-i}f}{(1-q)^m
q^{m(m-1)/2}z^m}.
\]
In particular, when $m=n$,  $D_z^n(f)=0$ for any $f$. And this
implies that
\[
f\star
g=\sum_{j=0}^{n-1}\frac{(\frac{i\hbar}{2})^j}{[j]_q!}D_z^j(f)\sigma_z^jD_{\bar{z}}^j(g)\gamma^j.
\]
\end{proposition}
\begin{proof}
We prove the identity by induction. When $m=1$, we have
\[
D_z(f)=\frac{f(z,\bar{z})-f(qz,
\bar{z})}{z-\gamma(z)}=\frac{f(z,\bar{z})-\sigma_z(f)(z,\bar{z})}{(1-q)z}.
\]
Assume that the above identity holds for $m$. Then for $m+1$,
\begin{eqnarray*}
D_z^{m+1}(f)&=&D_z(D_z^m(f))=D_z(\frac{\sum_{i=0}^m
(-1)^{m-i}\frac{[m]_q!}{[m-i]_q![i]_q!}q^{i(i-1)/2}\sigma_z^{m-i}f}{(1-q)^m
q^{m(m-1)/2}z^m})\\
&=&\sum_{i=0}^m(-1)^{m-i}\frac{[m]_q!}{[m-i]_q![i]_q!}\frac{q^{i(i-1)/2}}{(1-q)^m
q^{m(m-1)/2}}D_z(\frac{\sigma_z^{m-i}f}{z^m})\\
&=&\sum_{i=0}^m(-1)^{m-i}\frac{[m]_q!}{[m-i]_q![i]_q!}\frac{q^{i(i-1)/2}}{(1-q)^m
q^{m(m-1)/2}}\frac{\frac{\sigma_z^{m-i}f}{z^m}-\frac{\sigma_z^{m-i+1}(f)}{\gamma(z^m)}}{(z-\gamma(z))}\\
&=&\sum_{i=0}^m(-1)^{m-i}\frac{[m]_q!}{[m-i]_q![i]_q!}\frac{q^{i(i-1)/2}}{(1-q)^m
q^{m(m-1)/2}}\frac{q^m\sigma_z^{m-i}(f)-\sigma_z^{m-i+1}(f)}{q^m(1-q)z^{m+1}}\\
&=&\frac{1}{q^{m(m+1)/2}(1-q)^{m+1}z^{m+1}}\sum_{i=0}^m(-1)^{i}\frac{[m]_q!}{[m-i]_q![i]_q!}
q^{(m-i)(m-i-1)/2}(q^m\sigma_z^i(f)-\sigma_z^{i+1}(f))\\
&=&\frac{1}{q^{m(m+1)/2}(1-q)^{m+1}z^{m+1}}\Big(q^{m(m+1)/2}f+(-1)^{m+1}\sigma_z^{m+1}(f)+\\
&+&[m]_q!\sum_{i=1}^{m}(-1)^{i} (\frac{q^m
q^{(m-i)(m-i-1)/2}}{[m-i]_q![i]_q!}+\frac{q^{(m+1-i)(m-i)/2}}{[m-i+1]_q![i-1]_q!})\sigma_z^i(f)\Big)\\
&=&\frac{1}{q^{m(m+1)/2}(1-q)^{m+1}z^{m+1}}\Big(q^{m(m+1)/2}f+(-1)^{m+1}\sigma_z^{m+1}(f)+\\
&+&[m]_q!
\sum_{i=1}^m(-1)^i\frac{q^{(m+1-i)(m-i)/2}(q^i[m-i+1]_q!)+[i]_q!}{[m+1-i]_q![i]_q!}\sigma_z^i(f)\Big)\\
&=&\frac{1}{q^{m(m+1)/2}(1-q)^{m+1}z^{m+1}}\sum_{i=0}^{m+1}(-1)^iq^{(m+1-i)(m-i)/2}
\frac{[m+1]_q!}{[m+1-i]_q![i]_q!}\sigma_z^i(f).
\end{eqnarray*}
We have proved the identity of $D_z^m(f)$ by induction. To
conclude that $D_z^n(f)=0$. We see that by the above formula of
$D_z^n(f)$, as $[n]_q=0$,
\[
D_z^n(f)=\frac{1}{q^{n(n-1)/2}(1-q)^{n}z^n}(q^{(n-1)n/2}f+(-1)^n\sigma^n_z(f)).
\]
Since $\sigma_z^n(f)=f$, we have
\[
\frac{1}{q^{n(n-1)/2}(1-q)^{n}z^n}(q^{(n-1)n/2}+(-1)^n)(f).
\]
The statement follows from the identity $q^{(n-1)n/2}+(-1)^n=0.$
\end{proof}

We conclude from Proposition \ref{prop:D-x} that the star product
$\star$ on $S(\mathbb R^2)\rtimes \mathbb{Z}_n$ is convergent for
any value of $\hbar$.

In particular, when $n=2$, we have the following explicit formula of
a deformation on $S(\mathbb R^2)\rtimes \mathbb{Z}_2$
\[
f\star
g=fh+(\frac{i\hbar}{2})\frac{f(z,{\bar{z}})-f(-z,{\bar{z}})}{2z}\frac{f(-z,{\bar{z}})-f(-z,-{\bar{z}})}{2{\bar{z}}}.
\]

\begin{rk}
Here our formula of product uses ``normal ordering", by which we
mean that $D_z$ is contained only in the left component and
$D_{\bar{z}}$ is contained in the right component. We can also
define product with ``anti-normal ordering" or ``symmetric
ordering" as \cite{N}. The similar property like Proposition
\ref{prop:D-x} extends directly.
\end{rk}

\begin{rk}
We observe that the formula for the star product $\star$ on
$S(\mathbb R^2)\rtimes \mathbb{Z}_n[[\hbar]]$ works well for the
algebra $C^\infty_c(\mathbb R^2)\rtimes \mathbb{Z}_n[[\hbar]]$.
Again, $\star$ product is convergent for any two smooth functions
$f$ and $g$ on $\mathbb{R}^2$.
\end{rk}
\subsection{Deformation of singularity}
In this subsection, we compute the center of the above quantized
algebra $(S(\mathbb R^2)\rtimes \mathbb{Z}_n[[\hbar]], \star)$. We
prove that the center is closely connected to the deformation of the
underlying quotient space $\mathbb {R}^2/\mathbb{Z}_n$. We must say
that this kind of idea is already in \cite{E-G}. Here we are giving
concrete examples about this idea.

We write an element in $S(\mathbb R^2)\rtimes \mathbb{Z}_n[[\hbar]]$
by $\sum_{i=0}^{n-1}f_i\gamma^i$ with $f_i$ in $S(\mathbb
R^2)[[\hbar]]$.

\begin{proposition}\label{prop:center}
If $f=\sum_{i=0}^{n-1}f_i\gamma^i$ is in the center of $(S(\mathbb
R^2)\rtimes \mathbb{Z}_n[[\hbar]], \star)$, then $f$ is completely
determined by $f_0$ by the following formula
\[
\gamma(f_0)=f_0,\
f_j=(\frac{i\hbar}{2})^j\frac{D_z^j(f_0)}{[j]_q!(1-q^{-1})^j{\bar{z}}^j}=(-\frac{i\hbar}{2})^jq^{-j(j-1)/2}
\frac{\sigma_z^jD_{\bar{z}}^j(f_0)}{(1-q)^j[j]_q!z^j}, i=1,\cdots,
n-1.
\]
Therefore, as a vector space the center of the quantum algebra is
isomorphic to $S(\mathbb R^2)^{\mathbb{Z}_n}$, the algebra of
$\mathbb{Z}_n$ invariant polynomials on $V$.
\end{proposition}
\begin{proof}
We need to first prove the above two expressions for $f_i$ are
same. We prove this using Proposition \ref{prop:D-x}.
\begin{eqnarray*}
&&(-\frac{i\hbar}{2})^jq^{-j(j-1)/2}
\frac{\sigma_z^jD_{\bar{z}}^j(f_0)}{(1-q)^j[j]_q!z^j}\\
&=&\frac{(-\frac{i\hbar}{2})^jq^{-j(j-1)/2}}{(1-q)^j[j]_q!z^j}
\sigma_z^j\frac{\sum_{k=0}^j(-1)^{k}
\frac{[j]_{q^{-1}!}}{[j-k]_{q^{-1}}![k]_{q^{-1}}!}q^{-(j-k)(j-k-1)/2}
\sigma_{\bar{z}}^k(f_0)}{(1-q^{-1})^jq^{-j(j-1)/2}{\bar{z}}^j}\\
&=&\frac{(\frac{i\hbar}{2})^jq^{-j(j-1)/2}\sum_k
(-1)^{j-k}\frac{1}{[j-k]_q![j]_q!}q^{(j-k)(j-k-1)/2+k(k-1)/2-(j-k)(j-k-1)/2}
\sigma_z^j\sigma_{\bar{z}}^k(f_0)}{(1-q)^j(1-q^{-1})^jz^j{\bar{z}}^j}\\
&=&\frac{(\frac{i\hbar}{2})^jq^{-j(j-1)/2}}{[j]_q!(1-q^{-1})^j{\bar{z}}^j}
\frac{\sum_k(-1)^{j-k}\frac{[j]_q!}{[j-k]_q![k]_q!}q^{k(k-1)}\sigma_z^{j-k}f_0}{(1-q)^jz^j}
\end{eqnarray*}
where in the last line we have used that
$\sigma_z\sigma_{\bar{z}}(f_0)=\gamma(f_0)=f_0$. By Proposition
\ref{prop:D-x}, we conclude that the above expression is equal to
\[
\frac{(\frac{i\hbar}{2})^j}{[j]_q!(1-q^{-1})^j{\bar{z}}^j}D_z^j(f_0).
\]

Let $f=f_0+f_1\gamma+\cdots f_{n-1}\gamma^{n-1}$ be an element in
the center of $S(\mathbb R^2)\rtimes \mathbb Z_n$. We compute
$f\star z=\sum_{i}f_i\star
\gamma^i(z)\gamma^i=\sum_{i}q^if_iz\gamma^i$, and $z\star f=\sum_j
z\star f_j \gamma^j=(\sum_j zf_j
+\frac{i\hbar}{2}\sigma_zD_{\bar{z}}(f_j)\gamma)\gamma^j
=\sum_{j}(zf_j+\frac{i\hbar}{2}\sigma_zD_{\bar{z}}(f_{j-1})\gamma^j$.
As $f\star z=z\star f$,
$f_j=-\frac{i\hbar}{2}\sigma_zD_{\bar{z}}(f_{j-1})/[(1-q)[j]_q!z]$.
And we can solve by induction to find that $f_i$ has to be the form
expressed in the Proposition.

To prove that the above defined
$f=f_0+f_1\gamma+\cdots+f_{n-1}\gamma^{n-1}$ is in the center, we
show in the following $f\star g=g\star f$ for any $g\in S(\mathbb
R^2)$ and $\gamma(f)=f$.

For $\gamma(f)=f$, it is enough to prove that $\gamma(f_i)=f_i$.
This is obvious from the following identity
\[
f_j=\frac{(\frac{i\hbar}{2})^jq^{-j(j-1)/2}}{[j]_q!(1-q^{-1})^j{\bar{z}}^j}
\frac{\sum_k(-1)^{j-k}\frac{[j]_q!}{[j-k]_q![k]_q!}q^{k(k-1)}\sigma_z^{j-k}f_0}{(1-q)^jz^j}
\]
and the fact that $f_0$ is $\gamma$ invariant.

For $f\star g= g\star f$, we compute the two sides of the equation
separately.
\begin{eqnarray*}
f\star g&=&\sum_i f_i\gamma^i \star g=\sum_i f_i\star
\gamma^i(g)\gamma^i\\
&=& \sum_{i,j}
\frac{(\frac{i\hbar}{2})^j}{[j]_q!}D_z^j(f_i)\sigma_z^jD_z^j(\gamma^i(g))\gamma^{i+j}\\
&=&\sum_{j,k}\frac{(\frac{i\hbar}{2})^j}{[j]_q!}D_z^j\Big((\frac{i\hbar}{2})^k
\frac{D_z^k(f_0)}{[k]_q!(1-q^{-1})^k{\bar{z}}^k}\Big)\sigma_z^jD_{\bar{z}}^j(\gamma^k(g))\gamma^{j+k}\\
&=&\sum_{j,k}\frac{(\frac{i\hbar}{2})^{j+k}}{[k]_q![j]_q!}
\frac{D_z^{j+k}(f_0)\sigma_z^jD_{\bar{z}}^j(\gamma^j(g))\gamma^{j+k}}{(1-q^{-1})^k{\bar{z}}^k}\\
&=&\sum_{k}(\frac{i\hbar}{2})^{k}D_z^{k}(f_0)\sum_{j=0}^k
\frac{\sigma_z^{k-j}D_{\bar{z}}^{k-j}(\gamma^j(g))\gamma^{k}}{[j]_q![k-j]_q!(1-q^{-1})^j{\bar{z}}^j}.
\end{eqnarray*}
Applying Proposition \ref{prop:D-x} to the above
$D_{\bar{z}}^{k-j}$, we have that $f\star g$ is equal to
\begin{eqnarray*}
&=&\sum_{k}(\frac{i\hbar}{2})^{k}D_z^{k}(f_0)\sum_{j=0}^k
\frac{1}{[j]_q![k-j]_q!(1-q^{-1})^j{\bar{z}}^j}\\
&&\sigma_z^{k-j}\left(\frac{\sum_{l=0}^{k-j}(-1)^l\frac{[k-j]_{q^{-1}}!}{[k-j-l]_{q^{-1}}![l]_{q^{-1}}!}
q^{-(k-j-l)(k-j-l-1)/2}\sigma_{\bar{z}}^{l}(\sigma^j_z\sigma_{\bar{z}}^jg)}{(1-q^{-1})^{k-j}q^{-(k-j)(k-j-1)/2}{\bar{z}}^{k-j}}
\right)\gamma^{k}\\
&=&\sum_{k}(\frac{i\hbar}{2})^{k}D_z^{k}(f_0)\sum_{j=0}^k\sum_{l=0}^j
\frac{(-1)^{j-l}q^{(j-l)(j-l-1)/2}\sigma^k_z\sigma_{\bar{z}}^{j}g}{[l]_q![k-j]_{q}![j-l]_{q}!(1-q^{-1})^k{\bar{z}}^k}
\gamma^{k}.
\end{eqnarray*}
It is not difficult to check that
$\sum_{i=0}^j\frac{(-1)^{j-l}q^{(j-i)(j-i-1)/2}}{[i]_q![j]_q!}=0$ if
$j\ne 0$, and $=1$ if $j=0$. We replace this computation into the
above line and have
\[
f\star g=\sum\frac{(\frac{i\hbar}{2})^k}{[k]_q!(1-q^{-1})^k
{\bar{z}}^k}D_z^k(f_0)\sigma_z^kg \gamma^k.
\]
The computation of $g\star f$ is similar to the above and we
conclude that
\[
f\star g=\sum\frac{(\frac{i\hbar}{2})^k}{[k]_q!(1-q^{-1})^k
{\bar{z}}^k}D_z^k(f_0)\sigma_z^kg \gamma^k=g\star f.
\]
\end{proof}

\begin{rk}
The above proof on $f$ belonging to the center of quantum algebra
can be simplified by checking that $f$ commutes with the generators
of $S(\mathbb{R}^2)\rtimes \mathbb{Z}_n[[\hbar]]$, which consists of
$z,{\bar{z}},\gamma$. We have taken the above proof because it
extends to the algebra $C^\infty_c(\mathbb R^2)\rtimes
\mathbb{Z}_n[[\hbar]]$ directly.
\end{rk}

In the following, we study the algebraic structure on the center
$Z^\hbar_{\mathbb C}(\mathbb R^2,\mathbb{Z}_n)$ of the
(complexified) quantum algebra $\big(S(\mathbb R^2)\rtimes \mathbb
Z_n[[\hbar]]\big)\otimes_{\mathbb R}\mathbb C$. (The reason that we
consider the complexified algebra is that it is relatively easy to
write down a set of generators and their relations for the center.
However, our following discussion also works for the real algebra.)
It is easy to check that $u=z^n$, $v={\bar{z}}^n$ and
$w=z{\bar{z}}+\frac{\frac{i\hbar}{2}}{1-q^{-1}}\gamma$ are in the
center of the complexified quantum algebra $\big(S(\mathbb
R^2)\rtimes \mathbb Z_n[[\hbar]]\big)\otimes_{\mathbb R}\mathbb C$.
According to the isomorphism as vector space between the center of
the quantum algebra and $S(\mathbb R^2)^\Gamma$, we know that
$u=z^n$, $v={\bar{z}}^n$ and
$w=z{\bar{z}}+i\hbar\gamma/(2(1-q^{-1}))$ generates the whole center
$Z^\hbar_{\mathbb C}(\mathbb R^2,\mathbb{Z}_n)$. The relation
between these three generators is generated by
\begin{eqnarray*}
z^n \star {\bar{z}}^n&=&z\star \cdots \star z\star {\bar{z}}\star \cdots {\bar{z}}=z^{\star(n-1)}(z{\bar{z}}+\frac{i\hbar}{2}\gamma){\bar{z}}^{\star(n-1)}\\
&=&z^{\star(n-1)}(w+\frac{\frac{i\hbar}{2}}{1-q}\gamma){\bar{z}}^{\star(n-1)}\\
&=&z^{\star(n-1)}\star
{\bar{z}}^{\star(n-1)}\star(w+\frac{\frac{i\hbar}{2}
q^{-n+1}}{1-q}\gamma)\\
&=&z^{\star(n-2)}\star
{\bar{z}}^{\star(n-2)}\star(w+\frac{\frac{i\hbar}{2}
q^{-q+2}}{1-q}\gamma)\star (w+\frac{\frac{i\hbar}{2} q^{-n+1}}{1-q}\gamma)\\
&&\cdots \\
&=&(w+\frac{\frac{i\hbar}{2}}{1-q}\gamma)\star \cdots
\star(w+\frac{\frac{i\hbar}{2} q^{-n+1}}{1-q}\gamma)
\end{eqnarray*}
As $w$ is in the center, the last line can be viewed as the
expansion of $w^{\star
(n)}+\frac{(\frac{i\hbar}{2})^nq^{-\frac{n(n-1)}{2}}}{(1-q)^n}\gamma^n$.
Therefore,  $Z^\hbar_{\mathbb C}(\mathbb R^2,\mathbb{Z}_n)$ is
generated by $u,v,w$ with the relation that
$u^nv^n=w^n+(\frac{i\hbar}{2})^nq^{-\frac{n(n-1)}{2}}/(1-q)^n$. In
particular, we defines a deformation of the cone
$<u,v,w>/\{u^nv^n=w^n\}$, which is the algebra of polynomials on the
quotient $V/\mathbb{Z}_n$. Furthermore, we notice that for function
$F^\hbar(u,v,w)=u^nv^n-w^n-\frac{(\frac{i\hbar}{2})^nq^{-\frac{n(n-1)}{2}}}{(1-q)^n}$,
$(F^\hbar_u,F^\hbar_v,F^\hbar_w)$ is a non-zero vector in
$\mathbb{C}^3$ if and only if $u=v=w=0$, which is not on the surface
determined by $F^\hbar=0$. Therefore, we conclude that $F^\hbar=0$
determines a smooth surface when $\hbar\ne 0$, and $Z^\hbar_{\mathbb
C}(\mathbb R^2,\mathbb{Z}_n)$ is a nontrivial deformation of the
cone $\mathbb C^2/\mathbb{Z}_n$.

On the other hand, we look at the 0-th Poisson cohomology of the
Poisson structure $\pi_\gamma$ on $H^\bullet(S(\mathbb{R}^2)\rtimes
\mathbb{Z}_n,S(\mathbb{R}^2)\rtimes \mathbb Z_n )\otimes _{\mathbb
R}\mathbb{C}$. It is not difficult to see that the 0-th Poisson
cohomology $H^0_{\pi_\gamma}$ is isomorphic to $S(\mathbb
R^2)^{\mathbb{Z}_n}\otimes _{\mathbb R}\mathbb C$ as an algebra.

We summarize the above study into the following corollary.
\begin{proposition}The center $Z^\hbar_{\mathbb C}(\mathbb R^2,\mathbb{Z}_n)$ is not
isomorphic to the Poisson center
$H^0_{\pi_\gamma}[[\hbar]]=S(\mathbb R^2)^{\mathbb{Z}_n}[[\hbar]]$,
but defines a nontrivial deformation.
\end{proposition}

In particular, when $n=2$, the center $Z^\hbar_{\mathbb C}(\mathbb
R^2,\mathbb{Z}_2)$ is equal to
$<u,v,w>/\{uv=w^2+\frac{\hbar^2}{16}\}$. This is the algebra of
polynomial functions on the hyperboloid (when $\hbar$ is real, the
surface is one-sheeted, when $\hbar$ is imaginary, the surface is
two-sheeted.).

\begin{rk}
We can extend the above discussion of center to quantization of more
general Poisson structures. For example, the same discussion holds
true for the center of the quantization of the linear Poisson
structure $z\partial_x\wedge \partial_y$ on $S(\mathbb R^3)\times
\mathbb Z_n$ with $\mathbb Z_n$ acting on $\mathbb R^3$ by rotating
the $x,y$-plane and fixing the $z$ axis.
\end{rk}
In summary, we have seen that the center of the quantization of a
Poisson structure $\pi$ on an orbifold may not be isomorphic to the
0-th Poisson cohomology of $\pi$ as an algebra. On the other hand,
the well-known Duflo's isomorphism for a Lie algebra states that the
center of the universal enveloping algebra $\mathcal U(\mathfrak g)$
is isomorphic to the Poisson center of $S(\mathfrak g)$ as an
algebra. Our examples suggest that the natural extension of Duflo's
isomorphism does not hold in the case of quantization of Lie Poisson
structures on orbifolds. We plan to study this interesting phenomena
in future publications.

\end{document}